\documentclass[10pt]{amsart}
\usepackage{amsmath}
\usepackage{amsfonts}
\usepackage{amssymb}
\usepackage{amsthm}
\usepackage{mathrsfs}       
\usepackage{graphicx}     
\usepackage[utf8]{inputenc}
\allowdisplaybreaks

\usepackage{xspace}
\usepackage[colorlinks=true]{hyperref}

\newcommand{\Sph}{\mathbb{S}}
\newcommand{\R}{\mathbb{R}}
\newcommand{\N}{\mathbb{N}}

\newcommand{\Z}{\mathbb{Z}}
\newcommand{\comp}{\mathbb{C}}

\newcommand{\C}{\mathscr{C}}

\newcommand{\ud}{\,\mathrm{d}}

\let \Re \relax
\DeclareMathOperator{\Re}{Re}

\newcommand{\ovl}[1]{\overline{#1}}
\newcommand{\X}{\mathscr{X}}

\newcommand{\floor}[1]{\lfloor #1 \rfloor}

\newtheorem{thm}{THEOREM}[section]
\newtheorem{remark}[thm]{REMARK}
\newtheorem{lemma}[thm]{LEMMA}
\newtheorem{definition}[thm]{DEFINITION}
\newtheorem{proposition}[thm]{PROPOSITION}

\newcounter{thmbiss}

\title[Controllability of a 2D quantum particle in a varying domain]{Controllability of a 
2D quantum particle in a time-varying disc with radial data}

 \author[I.~Moyano]{Iv\'an Moyano}
\begin{address}{Centre de Robotique (CAOR), Mines-ParisTech \\
 60, Bld Saint-Michel,  
 75272, Paris, 
 France \\ 
 and
Laboratoire Jacques-Louis Lions. Universit\'e-Pierre-et-Marie-Curie (UPMC)\\
4, Pl. Jussieu,
75252 Paris, France. \\	 }
  \email{\href{mailto:ivan.moyano@mines-paristech.fr}{ivan.moyano@mines-paristech.fr}}
 \end{address}

\thanks{{\em Addresses:}
\\
Laboratoire Jacques-Louis Lions. Universit\'e-Pierre-et-Marie-Curie (UPMC)\\
4, pl. Jussieu, 
75252, Paris, France, \\
\\
Centre de Robotique (CAOR), Mines-ParisTech \\
60, Bld Saint-Michel,  
75272, Paris, 
France. \\
  }

\begin{document}

\begin{abstract}
In this article we consider a 2-D quantum particle confined a disc whose radius can be deformed continuously in time. 
We study the problem of controllability of such a quantum particle via deformations of the initial disc, i.e., when we set the time-dependent 
radius of the disc to be control variable. We prove that the resulting
system is locally controllable around some radial trajectories which are linear combinations of the first three radial 
eigenfunctions of the Laplacian in the unit disc with Dirichlet boundary conditions. We prove this result, thanks to the linearisation principle,
by studying the linearised system, 
which leads to a moment problem that can be solved using some results from Nonharmonic Fourier series. 
In particular, we have to deal with fine properties of Bessel functions. 
\end{abstract}

\maketitle


\tableofcontents

\section{Introduction}

\subsection{Physical background}
We consider a $d-$dimensional quantum particle, for $d \geq 1$, of mass $m$, under no external forces. 
According to Quantum Mechanics, the state of such a particle can be described by a complex-valued wave-function (see \cite[Secs 2.2.1, 2.2.3]{BasdevantDalibard})
\begin{equation*}
 \psi : \R^+ \times \R^d \rightarrow \comp,  \quad \textrm{ with   } \int_{\R^d} |\psi(t,x)|^2 \ud x = 1\footnote{The measure $|\psi(t,x)|^2\ud x$ is interpreted as a probability density, what explains the constraint.}, \quad \forall t \in \R^+,
\end{equation*} satisfying the Schr\"odinger equation
\begin{equation*}
i \partial_t \psi = -\frac{\hbar}{2m}\Delta_x \psi, \quad (t,x)\in \R^+ \times \R^d,  
\end{equation*} where $\hbar$ stands for normalised the Planck constant. In some instances (e.g., potential wells \cite[Sect.4.3.4]{BasdevantDalibard}), it is possible to confine 
the dynamics of a quantum particle within a region of the space, namely a regular open set $\Omega \subset \R^d$, 
which leads to a boundary-value problem for the associated wave-function, of the form
\begin{equation*}
\left\{ \begin{array}{ll}
i \partial_t \psi = -\frac{\hbar}{2m}\Delta_x \psi, & (t,x)\in \R^+ \times \Omega,  \\
\psi = 0, & (t,x) \in \R^+ \times \partial \Omega,
\end{array} \right.
\end{equation*} and the condition 
\begin{equation*}
\int_{\Omega} |\psi(t,x)|^2 \ud x = 1, \quad \forall t\in \R^+.
\end{equation*} This allows to consider a time-dependent confinement regions, namely a family of smooth open 
sets $\left\{ \Omega(t) \right\}_{t \geq 0}$, varying continuously with respect to time, within which the particle is confined. This question has attracted attention in Physics literature, as the works \cite{Doescher,Pinder,Makowski} or the survey \cite{Knobloch} account for. \par
In terms of the wave function, a quantum particle confined in $\left\{ \Omega(t) \right\}_{t \geq 0}$ 
must satisfy 
\begin{equation}
\int_{\Omega(t)} |\psi(t,x(t))|^2 \ud x =1, \quad \forall t \in \R^+,
\label{eq:conservationdelaprobabilite}
\end{equation}and the Schr\"odinger equation   
\begin{equation}
\left\{ \begin{array}{ll}
i \partial_t \psi = -\frac{\hbar}{2m}\Delta_{x(t)} \psi, & (t,x(t))\in \R^+ \times \Omega(t),  \\
\psi = 0, & (t,x(t)) \in \R^+ \times \partial \Omega(t),
\end{array} \right.
\label{eq:Schrodingergenerique}
\end{equation} which implies that a time-dependent boundary condition must be taken into account. In \cite{BMT, Munier} it is shown that, 
when $d=1$, such a system can be handled by a suitable change of
 variables, which transforms the original problem into a system posed on a fixed domain, with a time-dependent potential. 
 In particular, the works \cite{Miyashita, Aslangul} show that, even in the case $d=1$, keeping the particle
 confined in a time-varying box during the whole time-evolution can be extremely difficult, as some 
 unexpected instability phenomena can appear.  \par

From the perspective of Control of Partial Differential Equations, the seminal paper by P. Rouchon \cite{Rouchon} 
has raised the question of how to find a convenient family of deformations in order to control the dynamics 
of a confined quantum particle, for example to pass from the ground state to an 
excited state in a given time. This problem has been understood in one-dimensional situations, both for the 
linear free evolution (see \cite{KB08}) and the nonlinear regime describing a Bose-Einstein condensate (see \cite{KBHLHT13}). 
The goal of this work is to explore the same question in a two-dimensional setting.

\subsection{Controllability of a 2-D quantum particle confined in a disc via domain deformations}
Let, us consider for some $T^*>0$, $R\in \C^0([0,T^*];\R^+_*)$. We define the time-varying open discs
\begin{equation*}
D_{R(\tau)}:= \left\{ (z,w)\in \R^2; \, z^2 + w^2 < R(\tau)^2 \right\}, \quad \forall \tau \in [0,T^*],
\end{equation*} and we set the Schr\"{o}dinger equation on this variable domain, according to (\ref{eq:Schrodingergenerique}), which 
in adimensionalised form (i.e., we set $m=1$, $h=2$) reads
\begin{equation}
\left\{ \begin{array}{ll}
i\partial_{\tau} \phi = -\Delta_{z(\tau),w(\tau)} \phi, & (\tau,z,w) \in (0,T^*) \times D_{R(\tau)}, \\
\phi = 0 , & (\tau,z,w) \in (0,T^*) \times \partial D_{R(\tau)}. 
\end{array} \right.
\label{eq:Schrodinger}
\end{equation} 
\begin{remark}
An appropriate notion of solution of this problem will be defined in Section \ref{sec: WPSchrodinger}, thanks to a convenient change of variables, described in Section \ref{sec: cvarSchro}, that transforms (\ref{eq:Schrodinger}) into a system set on a fixed domain.   
\end{remark}
This is a control system whose state variable is the wave function $\phi(\tau,z,w)$, which, according to (\ref{eq:conservationdelaprobabilite}), must satisfy 
\begin{equation*}
\int_{D_{R(\tau)}} |\phi (\tau,z,w)|^2 \ud z \ud w = 1, \quad \forall \tau \in [0,T^*].
\end{equation*} We choose the time-dependent radius of the disc $D_{R(\tau)}$ to be the control variable, with the condition
\begin{equation}
R(0) = R(T^*) = 1.
\label{eq:initialandfinalradii}
\end{equation} We are interested in the following notion of controllability.

\begin{definition}[Controllability via domain transformations]
System (\ref{eq:Schrodinger}) is controllable in the space $\mathbb{X}$ if for any $\phi_0,\phi_f \in \mathbb{X}$, there exists $T^*>0$ and $R\in \C^0([0,T^*];\R^+_*)$ satisfying (\ref{eq:initialandfinalradii}) and such that the solution of (\ref{eq:Schrodinger}) with initial datum $\phi_{|t=0} = \phi_0$ satisfies $\phi_{|t=T^*} = \phi_f$.
\label{definition:controllability}
\end{definition}

The controllability of the Schr\"{o}dinger equation via domain transformations has been treated, in the one-dimensional case, by K. Beauchard in \cite{KB08}. 
The goal of this article is to explore the analogous question in the disc, as a first example of a two-dimensional case. Indeed, we shall prove a controllability result, 
according to Definition \ref{definition:controllability}, for regular enough radial data. \par 
More precisely, assuming that all data are radial, system (\ref{eq:Schrodinger}) writes
\begin{equation}
\left\{ \begin{array}{ll}
i \partial_{\tau} \phi = -\Delta_{\rho(\tau)} \phi,  & (\tau, \rho) \in (0,T^*) \times (0,R(\tau)), \\
\phi(\tau, R(\tau)) = 0, & \tau \in (0,T^*),
\end{array} \right.
\label{eq:systemsigmatilde}
\end{equation} where $\Delta_{\rho(\tau)} := \partial_{\rho(\tau)}^2 + \frac{1}{\rho(\tau)}  \partial_{\rho(\tau)} $ is the Laplacian operator in polar 
coordinates with radial data. 

\subsection{Change of variables}
\label{sec: cvarSchro}
Following  \cite{BMT,KB08,KBHLHT13}, let us introduce the new variables
\begin{equation}
\xi(t,r) := \phi(\tau, \rho), \textrm{ with } r:=\frac{\rho}{R(\tau)}, \, t:=\int_0^{\tau}\frac{\ud \sigma}{R(\sigma)^2},
\label{eq:changeofvariables} 
\end{equation} and the change of phase
\begin{equation*}
\psi(t,r) := \xi(t,r) \exp\left( -iu(t)r^2 + 4i \int_0^t u(s) \ud s \right),
\end{equation*} where 
\begin{equation}
u(t) := \frac{1}{4}\dot{R}(\tau)R(\tau), \quad \int_0^T u(s) \ud s.
\label{eq:control}
\end{equation} This change of variables transforms system (\ref{eq:systemsigmatilde}) into the following one, posed on a fixed domain,
\begin{equation}
\left\{ \begin{array}{ll}
i \partial_t \psi = -\Delta_r \psi + \left( \dot{u}(t) - 4u(t)^2 \right) r^2 \psi, & (t,r) \in (0,T)\times(0,1), \\
\psi(t,1) = 0, & t\in(0,T),
\end{array} \right.
\label{eq:systemsigma}
\end{equation} for $T:= \int_0^{T^*} \frac{\ud \sigma}{R(\sigma)^2} $ and 
\begin{equation}
\Delta_r := \partial_r^2 + \frac{1}{r}\partial_r. 
\label{eq:radialLaplacian}
\end{equation}System (\ref{eq:systemsigma}) is a bilinear control system in which the state is the function $\psi$ with $\psi(t) \in \Sph$, for any $t\in [0,T]$, where $\Sph$ is the unit sphere of $L^2(D;\comp)$, and the control is the real-valued function $u \in \dot{H}^1_0(0,T;\R)$, with  
\begin{equation*}
\dot{H}_0^1(0,T;\R) := \left\{ u \in H_0^1(0,T;\R), \, \int_0^T u(s) \ud s = 0  \right\}.
\end{equation*} Thanks to the change of variables described above, we find that the controllability of system (\ref{eq:systemsigma}) implies the controllability of system (\ref{eq:systemsigmatilde}), according to Definition \ref{definition:controllability}, via the application $u \mapsto R$. Indeed, this can be proved thanks to the following result (see \cite[Proposition 1]{KBHLHT13} for a proof).

\begin{proposition}[\cite{KBHLHT13}] Let $T>0$, $u \in L^{\infty}(0,T;\R)$ extended by zero in $(-\infty,0) \cup (T,\infty)$ and such that $\int_0^T u(s) \ud s =0$. The unique maximal solution of the Cauchy problem 
\begin{equation*}
\left\{ \begin{array}{ll}
g'(\tau) = 4e^{-2 \int_0^{g(\tau)} u(s) \ud s}, \\
g(0) = 0,
\end{array} \right.
\end{equation*} is defined for every $\tau \geq 0$, strictly increasing and satisfies 
\begin{equation*}
\lim_{\tau \rightarrow \infty} g(\tau) = + \infty.
\end{equation*} Thus, $T^* = g^{-1}(T)$ is well-defined and if $R$ is defined by 
\begin{equation*}
R(\tau) := e^{\int_0^{g(\tau)} u(s) \ud s},
\end{equation*} then (\ref{eq:initialandfinalradii}) and (\ref{eq:control}) are satisfied.
\label{prop:surjectivity}
\end{proposition}

%%%%%%%%%%%%%%%%%%%%%%%%%%%%%%%%%%%%%%%%%%%%%%%%%%%%%%%%%%%%%%%%%%%%%%%%%%%%%%%%%%%%%%%%%%%%%%%%%%%%%%%%%
%%%%%%%%%%%%%%%%%%%             FUNCTIONAL SETTING           %%%%%%%%%%%%%%%%%%%%%%%%%%%%%%%%%%%%%%%%%%%%
%%%%%%%%%%%%%%%%%%%%%%%%%%%%%%%%%%%%%%%%%%%%%%%%%%%%%%%%%%%%%%%%%%%%%%%%%%%%%%%%%%%%%%%%%%%%%%%%%%%%%%%%%

\subsection{Functional setting}
Let $D$ be the unit disc of $\R^2$. We shall work on the space $L^2(D;\comp)$, with the scalar product
\begin{equation}
\langle f, g \rangle_{L^2(D)} := \int_D f(x,y) \overline{g(x,y)} \ud x \ud y, \quad \forall f,g \in L^2(D;\comp).
\label{eq:scalarproduct}
\end{equation} Let $(A,D(A))$ be the operator defined by 
\begin{equation}
\left\{ \begin{array}{ll}
D(A):= H^2 \cap H_0^1(D;\comp), \\
A\psi := -\Delta \psi, \quad \forall \psi \in D(A).
\end{array} \right.
\label{eq:operatorLaplacian}
\end{equation} Let us recall that the eigenfunctions of this operator write, in polar coordinates, as follows (\cite[Ch.6, p.130]{Davies}) 
\begin{equation}
\varphi_{\nu,k}(r,\theta):= \frac{J_{\nu}(j_{\nu,k}r) e^{ik\theta}}{\sqrt{\pi}|J_{\nu+1}(j_{\nu,k})|} , \quad \forall (r,\theta)\in [0,1]\times[0,2\pi),  
\label{eq:eigenfunctions}
\end{equation} for every $(\nu,k)\in \N\times\N^*$, with eigenvalues 
\begin{equation}
\lambda_{\nu,k}:= j_{\nu,k}^2, \quad \forall (\nu,k)\in \N \times \N^*,
\label{eq:eigenvalues}
\end{equation} where $J_{\nu}$ is the Bessel function of the first kind and order $\nu\geq 0$ and $\left\{  j_{\nu,k} \right\}_{k\in \N^*}$ 
is the sequence of its zeros (see Appendix \ref{sec: appendixBessel} for details and notation). \par
Since the radial case will be particularly important in this article, we shall note, for simplicity,
\begin{equation}
\varphi_k: = \varphi_{0,k}, \quad \lambda_k:= \lambda_{0,k}, \quad \forall k \in \N^*.
\label{eq:radialeigenvectors}
\end{equation} Thus, from (\ref{eq:radialLaplacian}), one has
\begin{equation*}
-\Delta_r \varphi_k = \lambda_k \varphi_k, \quad \forall k \in \N^*.
\end{equation*} According to (\ref{eq:operatorLaplacian}), we introduce the spaces
\begin{equation*}
H^s_{(0)}(D;\comp) := D(A^{\frac{s}{2}}), \quad \forall s>0,
\end{equation*} endowed with the norm
\begin{equation}
\|f\|_{H^s_{(0)}} := \left(  \sum_{(\nu,k)\in \N\times \N^*} |j_{\nu,k}^s \langle f, \varphi_{\nu,k} \rangle_{L^2(D)} |^2 \right)^{\frac{1}{2}}, \quad \forall f \in H_{(0)}^s(D;\comp),
\label{eq:norms}
\end{equation} where $\langle \cdot, \cdot \rangle_{L^2(D)} $ is given by (\ref{eq:scalarproduct}). In the case $s=1$, we simply write $H_0^1(D)$, 
as usual, as well as $H^{-1}(D)$ for its dual space. In the radial case, we set
\begin{equation*}
H_{(0),rad}^s (D;\comp) := \left\{ f \in H_{(0)}^s(D;\comp); \, f \textrm{ is radial} \right\}, \quad \forall s> 0, 
\end{equation*} and $L^2_{rad}(D)$ when $s=0$. Furthermore, if $f  \in H_{(0),rad}^s(D;\comp)$, by changing variables, the norm (\ref{eq:norms}) reduces to 
\begin{equation*}
\|f\|_{H^s_{(0),rad}} := \left(  \sum_{k=1}^{\infty} |j_k^s \langle f, \varphi_k \rangle |^2 \right)^{\frac{1}{2}},
\end{equation*}up to a universal constant, with the scalar product
\begin{equation}
\langle f, g \rangle := \int_0^1 f(r) \ovl{g(r)} r \ud r, \quad \forall f,g \in L^2(D) \textrm{ radial}.
\end{equation} We observe that $H_{(0),rad}^s $ is a closed subspace of $H_{(0)}^s$. \par

\subsection{Main result}

The main result of this article is a local exact controllability result of system (\ref{eq:systemsigma}) around a well-chosen trajectory. 
To describe these states, let us introduce the set
\begin{equation}
\mathcal{D}:= \left\{ (\theta_2,\theta_3) \in \R^2; \quad  \theta_2, \theta_3 >0, \, \theta_2 + \theta_3 < 1    \right\},
\label{eq:setomegas}
\end{equation} and the family of states
\begin{equation}
\varphi^{\sharp} := \sqrt{1-\theta_2 -\theta_3} \varphi_1 + \sqrt{\theta_2}\varphi_2 + \sqrt{\theta_3} \varphi_3, \quad (\theta_2, \theta_3) \in \mathcal{D},
\label{eq:phidiege}
\end{equation} according to (\ref{eq:radialeigenvectors}). In this setting, we consider the associated wave packets
\begin{equation}
\psi^{\sharp}_{\tau} := e^{-i\lambda_1 \tau} \sqrt{1- \theta_2 - \theta_3} \varphi_1 + e^{-i\lambda_2 \tau}\sqrt{\theta_2} \varphi_2 + e^{-i\lambda_3 \tau} \sqrt{\theta_3} \varphi_3, \quad \tau \geq 0. 
\label{eq:psidiege}
\end{equation} Thus, $\psi^{\sharp}_0 = \varphi^{\sharp}$ and let $\psi^{\sharp}_t = e^{-it\Delta}\varphi^{\sharp}$, for $t \geq 0$. \par 
With this notation, the main result of this article is the following one.
\begin{thm}
Let $T>0$. There exists $\delta >0$ and a $\C^1-$map 
\begin{equation*}
\Gamma: V_0 \times V_T \rightarrow \dot{H}_0^1 (0,T;\R),
\end{equation*} where 
\begin{align}
V_0 & := \left\{ \psi_0 \in \Sph \cap H_{(0),rad}^3(D;\comp); \, \| \psi_0 - \varphi^{\sharp}  \|_{H_{(0)}^3} < \delta  \right\}, \label{eq:V0}\\
V_T & := \left\{ \psi_f \in \Sph \cap H_{(0),rad}^3(D;\comp); \, \| \psi_f - \psi_{T}^{\sharp}  \|_{H_{(0)}^3} < \delta  \right\}, \label{eq:VT} 
\end{align} such that $\Gamma(\varphi^{\sharp},\psi_{T}^{\sharp})=0$ and for any $(\psi_0,\psi_f) \in V_0 \times V_T$, the solution of (\ref{eq:systemsigma}) with $\psi_{|t=0} = \psi_0$ and control $u = \Gamma(\psi_0,\psi_f)$ satisfies 
\begin{equation*}
\psi_{|t=T} = \psi_f.
\end{equation*}
\label{thm:controllability2d}
\end{thm}

\begin{remark}
The choice of the states (\ref{eq:phidiege}) will be clear un Section \ref{sec: Schrolinearisecontrolable}, as the choice of 
more straightforward trajectories may lead to a 
linearised system which is not controllable (see Remark \ref{remark:choixdesetatsnaturels} for more details).
\end{remark}

%%%%%%%%%%%%%%%%%%%%%%%%%%%%%%%%%%%%%%%%%%%%%%%%%%%%%%%%%%%%%%%%%%%%%%%%%%%%%%%%%%%
%%%%%%%%%%%%%%%%%                       PREVOIUS RESULTS

\subsection{Previous work}

The problem of the controllability of a confined quantum particle via domain deformations has been 
solved in the one-dimensional case by K. Beauchard in \cite{KB08} for the free evolution and by 
K. Beauchard, H. Lange and H. Teismann in \cite{KBHLHT13} for a Bose-Einstein condensate. 
In both cases, the problem of controllability via domain transformations can be handled thanks
to a suitable change of variables, which reduces the problem into a bilinear control system under constraints. 
Let us point out that in 1-D the bilinear control of the Schr\"odinger equation has received much attention (we can mention the works 
\cite{Ball, Turinici, Boscain, KB05, KBCLhautes} among others).  
In particular, the techniques developed by K. Beauchard and C. Laurent in \cite{KBCL10} allow to prove local exact controllability 
results thanks to the Inverse Mapping Theorem and a certain smoothing effect. Let us observe that this approach simplifies the original 
proofs in \cite{KB08}, which use the Nash-Moser theorem. \par

Let us point out that, contrarily to the 1D case, in the 2D setting, much less results on bilinear control are known 
(see \cite{KBCL2D,Puelbilineaire}). Let us emphasise that, in particular, the results of \cite{Puelbilineaire} cannot be
applied to our case because of the geometric constraint (\ref{eq:initialandfinalradii}), which imposes a restriction in the control (see \ref{eq:control})  
for the bilinear problem (\ref{eq:systemsigma}). We refer to remark \ref{remark:choixdesetatsnaturels} for more details. 
Consequently, the controllability problem 
via domain deformations in a 2D setting is technically much more involved that in 1D, 
as the geometry of the deformations plays a major role. In the case of the disc, 
we can handle this by exploiting some fine properties of Bessel functions, through spectral decompositions (see Section 
\ref{sec: Schrolinearisecontrolable} for details), which is the major novelty of our work. 
Indeed, this article represents the very first step in the exploration of the 2D problem, which should lead to other developments in the future (see Section \ref{sec: Schrocomments} for comments and perspectives).

%%%%%%%%%%%%%%%%%%%%%%%%%%%%%%%%%%%%%%%%%%%%%%%%%%%%%%%%%%%%%%%%%%%%%%%%%%%%%%%%%%%%%%%%%%%%%5
%%%%%%%%%%%%           STRATEGY      %%%%%%%%%%%%%%%%%%%%%%%%%%%%%%%%%%%%%%%%%%%%%%%%%%%%%%%%

\subsection{Strategy and outline}

In this work, we shall follow the linearisation principle to prove a local exact controllability result, exploiting the connection with bilinear control problems under constraints, in the spirit of 
\cite{KB08,KBCL10}. More precisely, the strategy of the proof of Theorem \ref{thm:controllability2d} 
has three main ingredients: 
\begin{itemize}
\item we prove first that the linearised system around $(\psi^{\sharp}_t, u\equiv 0)$ is controllable,
\item secondly, we prove that the end-point map (see Section \ref{sec: endpointmapC1} for the definition) is of class $\C^1$ between some adequate spaces,
\item finally, we deduce the local exact controllability from the Inverse Mapping Theorem.
\end{itemize}

\subsubsection{Outline of the article}
In Section \ref{sec: WPSchrodinger} we recall the well-posedness of system (\ref{eq:systemsigma}) and 
state a smoothing effect. In Section \ref{sec: endpointmapC1} we use the smoothing effect to prove that 
the end-point map is of class $\C^1$. In Section \ref{sec: Schrolinearisecontrolable} we show that the linearised 
system around $(\psi^{\sharp}_t,u\equiv0)$ is controllable, thanks to the resolution of a suitable moment problem, that can be solved
through the construction of an adequate Riesz basis in Section \ref{sec: constructionofaRieszbasisSchrodinger}
and a key asymptotic result proven in Section \ref{sec: preuveformuleBessel}. 
In Section \ref{sec: Schroconclusion} we conclude the proof of Theorem \ref{thm:controllability2d} thanks to the Inverse Mapping theorem. 
In Section \ref{sec: Schrocomments} we gather some comments and perspectives. In Appendix \ref{sec: appendixBessel} we gather some results on Bessel function. 
In Appendix \ref{sec: appendixMoment} we gather some results on abstract and trigonometric moment problems that are useful 
in Section \ref{sec: Schrolinearisecontrolable}.

%%%%%%%%%%%%%%%%%%%%%%%%%%%%%%%%%%%%%%%%%%%%%%%%%%%%%%%%%%%%%%%%%%%%%%%%%%%%%%%%%%%%%%%%%%%%
%%%%%%%%%%%%%            WP AND SMOOTHING        %%%%%%%%%%%%%%%%%%%%%%%%%%%%%%%%%%%%%%%%%%%

\section{Well-posedness and smoothing effect}
\label{sec: WPSchrodinger}

The goal of this section is to prove a well-posedness result in an appropriate functional setting for the system
\begin{equation}
\left\{ \begin{array}{ll}
i \partial_t \psi = -\Delta_r \psi + u(t) r^2 \psi + f(t,r), & (t,r) \in (0,T) \times (0,1), \\
\psi(t,1) = 0, & t\in (0,T), 
\end{array} \right.
\label{eq:Cauchyproblem}
\end{equation} where $\Delta_r$ is given by (\ref{eq:radialLaplacian}). \par 
Let us recall that the Schr\"{o}dinger operator $iA$, where $A$ is given by (\ref{eq:operatorLaplacian}), generates a group of isometries in $H^s_{(0)}(D;\comp)$, for $s\geq 0$, that we denote $\left( e^{-it\Delta}\right)_{t\geq 0}$. Furthermore, thanks to (\ref{eq:radialeigenvectors}), for any $\psi_0 \in H^s_{(0),rad}(D;\comp)$, one has 
\begin{equation}
e^{-it\Delta} \psi_0 : = \sum_{k=1}^{\infty} e^{-i\lambda_k t } \langle \psi_0, \varphi_k \rangle \varphi_k.
\label{eq:evolutionlibre}
\end{equation} 

\begin{proposition}
Let $T>0$. For every $\psi_0 \in H_{(0),rad}^3(D)$, $f \in L^2(0,T;H^3 \cap H_{0,rad}^1(D))$, $u \in L^2((0,T);\R)$, there exists a unique weak solution of system (\ref{eq:Cauchyproblem}) with $\psi_{|t=0} = \psi_0$, i.e., $\psi \in \C^0([0,T];H^3_{(0),rad}(D))$ such that 
\begin{equation}
\psi(t) = e^{-i\Delta t}\psi_0 + i \int_0^t e^{-i\Delta (t-s)} \left[ u(s) r^2 \psi + f(s,r)  \right] \ud s, \quad  \forall t\in [0,T].
\label{eq:Duhamel}
\end{equation} Furthermore, for every $M>0$ there exists a constant $C_1 = C_1(T,M)>0$, such that if $\|u\|_{L^2(0,T)}<M$, then
\begin{equation}
\| \psi \|_{\C^0([0,T];H^3_{(0),rad})} \leq C_1(T,M) \left( \| \psi_0 \|_{H^3_{(0),rad}} + \| f \|_{L^2(0,T;H^3 \cap H^1_{0,rad})}   \right),
\label{eq:energyestimate}
\end{equation} and such that $C_1(t,M)$ is uniformly bounded on any bounded interval. Moreover, if $f=0$, we have 
\begin{equation}
\|\psi(t)\|_{L^2(D)} = \|\psi_0 \|_{L^2(D)}, \quad  \forall t\in [0,T].
\label{eq:conservationnormeL2}
\end{equation}
\label{prop:WP}
\end{proposition} The proof of this result relies on the smoothing effect of next section.

\subsection{Smoothing effect}
\label{sec: Schrosmoothing}
As it was shown by K. Beauchard and C. Laurent in the one-dimensional case in \cite[Proposition 2]{KBCL10}, a certain smoothing effect can be expected for $\left( e^{-it\Delta} \right)_{t\geq 0}$ in a suitable functional framework. This has been extended to a large class of smooth domains in any space dimension by J.P. Puel in \cite{Puel}. To be precise, in the case of the unit disc $D\subset \R^2$, let 
\begin{equation}
\left\{ \begin{array}{ll}
i\partial_t \psi = -\Delta \psi + f(t,x,y), & (t,x,y) \in (0,T)\times D, \\
\psi = 0, & (t,x,y) \in (0,T)\times \partial D.
\end{array} \right.
\label{eq:Schrodingersource}
\end{equation} Then, the following  has been proved (see \cite[Theorem 2.1]{Puel}).

\begin{proposition}[\cite{Puel}]
Let $T>0$. For every $\psi_0 \in H_{(0)}^3(D)$ and for every $f = g + h$, where
\begin{equation}
g \in L^1(0,T;H^3_{(0)}(D))
\label{eq:sourcebonne}
\end{equation} and 
\begin{equation}
h \in L^2(0,T;H^2\cap H_0^1(D)), \quad \Delta^2 h =0, \quad \Delta h|_{\partial D} \in L^2(0,T;L^2(\partial D)),
\label{eq:sourcemechante}
\end{equation} the solution of (\ref{eq:Schrodingersource}) with $\psi_{|t=0}= \psi_0$ satisfies $\psi \in \C^0([0,T];H_{(0)}^3(D))$ and there exists a constant $C>0$, independent of $\psi_0$, $g$ or $h$, such that 
\begin{equation}
\|\psi\|_{\C^0([0,T];H_{(0)}^3)} \leq C\left( \|\psi_0\|_{H_{(0)}^3} + \|g\|_{L^1(0,T;H_{(0)}^3)} + \| \Delta h|_{\partial D} \|_{L^2(0,T;L^2(\partial D))} \right).
\label{eq:Puelenergyestimate}
\end{equation}
\label{prop:Puel}
\end{proposition}

\begin{proof}[Proof of Proposition \ref{prop:WP}.] Let $T>0$, $\psi_0 \in H_{(0),rad}^3(D)$ and $f \in L^2(0,T;H^3\cap H_{0,rad}^1(D))$. We consider the map
\begin{equation}
\left| \begin{array}{cccc}
F: & \C^0([0,T];H^3_{(0),rad}(D)) & \rightarrow & \C^0([0,T];H^3_{(0),rad}(D)) \\
   &  \psi   & \mapsto & \xi,
\end{array} \right.
\label{eq:fixedpoint}
\end{equation} where $\xi$ is the solution of
\begin{equation}
\left\{ \begin{array}{ll}
i\partial_t \xi(t,r) = -\Delta_r \xi(t,r) + u(t) r^2 \psi(t,r) + f(t,r), & (t,r) \in (0,T) \times (0,1), \\
\xi(t,1) = 0, & t \in (0,T), \\
\xi(0,r) = \psi_0(r), & r\in (0,1).
\end{array} \right.
\label{eq:fixedpointsystem}
\end{equation} Our aim is to prove that this map has a fixed point. We divide the proof in several steps.

\vspace{0.5em}
\textit{Step 1.} We show that (\ref{eq:fixedpoint}) is well-defined. \par 
By direct computation, we observe that, as $u\in L^2(0,T;\R)$, for every $\psi \in \C^0([0,T];H^3_{(0),rad}(D))$, we have $ u(t) r^2 \psi \in L^2(0,T;H^3\cap H^1_{(0),rad}(D))$. As a result, $\tilde{f} := u(t)r^2 \psi + f$ belongs to $ L^2(0,T;H^3\cap H^1_{(0),rad}(D))$. \par 
We can decompose $\tilde{f}$ as in Proposition {\ref{prop:Puel}}. Indeed, let us consider, for a.e. $t\in (0,T)$ the following elliptic problem
\begin{equation}
\left\{ \begin{array}{ll}
\Delta^2 g(t)  = \Delta^2 \tilde{f}(t), & \textrm{ in } D, \\
g(t) = \Delta g(t) = 0, & \textrm{ on } \partial D,
\end{array} \right.
\label{eq:ellipticproblem}
\end{equation} where $\Delta^2$ stands for the Bilaplacian operator. Since $\Delta^2 \tilde{f}(t) \in H^{-1}(D) $ for a.e. $t\in (0,T)$, by elliptic regularity results (see \cite[Th. 5.1, p. 166]{Lions}), we deduce
\begin{equation}
g \in L^2(0,T;H^3_{(0)}(D)).
\label{eq:hypothesesourcebonne}
\end{equation} Let us define next
\begin{equation}
h:= \tilde{f} - g.
\label{eq:decomposition}
\end{equation} Since $\Delta \tilde{f} \in L^2(0,T;H^1(D))$, and using (\ref{eq:hypothesesourcebonne}), we have
\begin{eqnarray}
&& h\in L^2(0,T;H^2 \cap H_0^1(D)),  \label{eq:hypotesesourcemechante1} \\
&& \Delta_r h \in L^2(0,T;H^1(D)). \label{eq:hypothesesourcemechante2}
\end{eqnarray} Hence, from (\ref{eq:ellipticproblem}), 
\begin{equation}
\Delta^2 h(t)|_{\partial D} = 0  \textrm{ and } h(t)|_{\partial D} = 0, \quad \textrm{a.e } t \in (0,T).
\label{eq:hypothesesourcemechante3}
\end{equation} Moreover, using trace results (see \cite[Th.8.3, p.44]{Lions}), (\ref{eq:hypothesesourcemechante2}) implies 
\begin{equation}
\Delta h \in L^2(0,T;L^2(\partial D)).
\label{eq:hypothesesourcemechante4}
\end{equation} Thanks to (\ref{eq:decomposition}), (\ref{eq:hypothesesourcebonne}) and (\ref{eq:hypotesesourcemechante1})--(\ref{eq:hypothesesourcemechante4}), we can apply Proposition \ref{prop:Puel} to system (\ref{eq:fixedpointsystem}). This implies in particular that, as all data are radial and $\Delta^2$ is invariant by rotations, we deduce $\xi \in \C^0([0,T];H^3_{(0),rad}(D))$.

\vspace{0.5em}
\textit{Step 2.} We derive an appropriate energy estimate for system (\ref{eq:fixedpointsystem}). We claim that 
\begin{equation}
\|\xi\|_{\C^0([0,T];H^3_{(0),rad})} \leq C(T) \left( \|\psi_0\|_{H^3_{(0),rad}} +  \|\tilde{f}\|_{L^2(0,T;H^3\cap H_{(0),rad}^1)} \right),
\label{eq:auxiliaryenergyestimate}
\end{equation} for some constant $C(T)>0$ which is bounded on bounded intervals $(0,T)$. \par
Indeed, according to (\ref{eq:Puelenergyestimate}), we have
\begin{eqnarray}
&& \|\xi\|_{\C^0([0,T];H^3_{(0),rad})} \nonumber \\
&& \quad \leq C \left( \|\psi_0\|_{H^3_{(0),rad}} + \|g\|_{L^1(0,T;H^3_{(0),rad})} + \|\Delta h_{|\partial D}\|_{L^2(0,T)} \right). \nonumber
\end{eqnarray} We treat the two last terms separately. For the first one, we observe that, using (\ref{eq:ellipticproblem}), elliptic regularity (see \cite[Th. 5.1, p. 166]{Lions}) and the Cauchy-Schwarz inequality, it follows

\begin{eqnarray}
\|g\|_{L^1(0,T;H^3_{0,rad})} &\leq & C_1 \|\Delta^2 \tilde{f}\|_{L^1(0,T,H^{-1})} \nonumber \\
& \leq & C_2 \|\tilde{f}\|_{L^1(0,T;H^3\cap H_{0,rad}^1)} \nonumber \\
& \leq & C_3 \sqrt{T}\|\tilde{f}\|_{L^2(0,T;H^3 \cap H_{0,rad}^1)}. \nonumber
\end{eqnarray} For the other term, using (\ref{eq:ellipticproblem}), (\ref{eq:hypothesesourcemechante3}) and the continuity of the trace map (see \cite[Th.8.3, p.44]{Lions}),
\begin{eqnarray}
\|\Delta h_{|\partial D}\|_{L^2(0,T)}  &=& \|\Delta \tilde{f} _{|\partial D} \|_{L^2(0,T)} \nonumber \\
& \leq & C_4 \|\Delta \tilde{f}\|_{L^2(0,T;H^1_{rad}(D))} \nonumber \\
&\leq &  C_5 \| \tilde{f} \|_{L^2(0,T;H^3\cap H_{(0),rad}^1(D)}). \nonumber
\end{eqnarray} Putting these estimates together, we obtain (\ref{eq:auxiliaryenergyestimate}).

\vspace{0.5em}
\textit{Step 3.} We show that $F$ is a contraction in $\C^0([0,T]; H^3_{(0),rad}(D))$. \par
Let $\psi_1, \psi_2 \in \C^0([0,T]; H^3_{(0),rad}(D))$. Then, by linearity of system (\ref{eq:fixedpointsystem}), $\eta:= \psi_1 - \psi_2 $ satisfies 
\begin{equation}
\left\{ \begin{array}{ll}
i\partial_t \eta(t,r) = -\Delta_r \eta(t,r) + u(t) r^2 (\psi_1 - \psi_2)(t,r), & (t,r) \in (0,T) \times (0,1), \\
\eta(t,1) = 0, & t \in (0,T), \\
\eta(0,r) = 0, & r\in (0,1).
\end{array} \right.
\label{eq:contractionsystem}
\end{equation} Using (\ref{eq:auxiliaryenergyestimate}), we deduce
\begin{align*}
& \left\| F[\psi_1] - F[\psi_2] \right\|_{\C^0([0,T];H^3_{(0),rad}(D))} \\
& \quad \quad \quad = \| \eta \|_{\C^0([0,T];H^3_{(0),rad}(D))} \\
& \quad \quad \quad \leq C(T) \| u(t)r^2 (\psi_1 - \psi_2) \|_{L^2(0,T;H^3 \cap H_{0,rad}^1(D))} \\
& \quad \quad \quad \leq C'(T) \|u\|_{L^2(0,T)} \| r^2 (\psi_1 - \psi_2) \|_{L^{\infty}(0,T,H^3 \cap H_{0,rad}^1(D))} \\
& \quad \quad \quad \leq C''(T) \|u\|_{L^2(0,T)} \| \psi_1 - \psi_2 \|_{\C^0([0,T];H^3_{(0),rad}(D))},
\end{align*} where $C''(T)>0$ is a constant which remains bounded on bounded intervals. \par

If $C''(T)\|u \|_{L^2} < 1$, this estimate shows that $F$ is a contraction in the Banach space $\C^0([0,T];H^3_{(0),rad}(D))$,
as $H^3_{(0),rad}(D)$ is closed in $H^3_{(0)}(D)$. The Banach fixed-point theorem gives then the existence of a unique fixed point of $F$.
Moreover, (\ref{eq:auxiliaryenergyestimate}) gives (\ref{eq:energyestimate}) in this case. \par

In order to extend the result to arbitrary $u \in L^2(0,T;\R)$, we choose $N\in \N^*$ and a partition of $[0,T]$,
namely $0=T_0 <T_1 < \dots < T_N =T$ and such that $\|u \|_{L^2(T_i,T_{i+1})}$ is small enough $\forall i \in \left\{ 1,\dots,N  \right\}$.
We then apply the preceding arguments in each interval $[T_i,T_{i+1}]$.\par 

Finally, whenever $f \equiv 0$ and $u\in \C^0([0,T];\R)$, identity (\ref{eq:conservationnormeL2})
follows by classical arguments. This allows to extend (\ref{eq:conservationnormeL2}) to the case $u\in L^2(0,T;\R)$ by density.

\end{proof}

%%%%%%%%%%%%%%%%%%%%%%%%%%%%%%%%%%%%%%%%%%%%%%%%%%%%%%%%%%%%%%%%%%%%%%%%%%%%%%%%%%%%%%%%%%%%%%
%%%%%%%%%%%%%%%%%%%%%%%%%%%%%%%%%%%%%%%%%%%%%%%%%%%%%%%%%%%%%%%%%%%%%%%%%%%%%%%%%%%%%%%%%%%%%%
%%%%%%%%%%%%%%%%%%%%%%%%%%%%%%%%%%%%%%%%%%%%%%%%%%%%%%%%%%%%%%%%%%%%%%%%%%%%%%%%%%%%%%%%%%%%%%
%%%%%%%%%%%%%%%%%              LOCAL EXACT CONTROLLABILITY    
%%%%%%%%%%%%%%%%%%%%%%%%%%%%%%%%%%%%%%%%%%%%%%%%%%%%%%%%%%%%%%%%%%%%%%%%%%%%%%%%%%%%%%%%%%%%%%

%%%%%%%%%%%%%%%%%%%%%%%%%%%%%%%%%%%%%%%%%%%%%%%%%%%%%%%%%%%%%%%%%%%%%%%%%%%%%%%%%%%%%%%%%%%%%%%%
%%%%%%%%%%%%%%%%%              C1 MAP 

\section{ $\C^1$-regularity of the end-point map}
\label{sec: endpointmapC1}

In order to define the end-point map, we shall need the following definitions. Let, for $s>0$,   
\begin{equation}
\X^s := H^s_{(0),rad}(D;\comp) \cap \Sph.
\label{eq:espaceX}
\end{equation} Setting $T>0$, let us fix $\xi \in \Sph$ and let us consider the tangent space
\begin{equation}
T_{\xi} \Sph := \left\{ f \in L^2(D;\comp); \, \Re \langle f, \xi \rangle_{L^2(D)} = 0 \right\}.
\label{eq:tangent} 
\end{equation} Then, we consider, thanks to Proposition \ref{prop:WP}, the end-point map
\begin{equation}
\left| \begin{array}{cccc}
\Theta_T: & \dot{H}_0^1(0,T;\R) \times \X^3 & \rightarrow & \X^3 \times \X^3, \\
		  & (u,\psi_0) & \mapsto & (\psi_0, \psi_{|t=T}),  
\end{array} \right.
\label{eq:endpointmap}
\end{equation} where $\psi$ is the solution of (\ref{eq:systemsigma}) with control $u$ and initial condition $\psi_0$. \par 
Let 
\begin{equation}
\mathcal{X}_0:= H^3_{(0),rad}(D;\comp) \cap T_{\varphi^{\sharp}}\Sph, \quad \mathcal{X}_T:= H^3_{(0),rad}(D;\comp) \cap T_{\psi^{\sharp}_T} \Sph. 
\label{X0T}
\end{equation} Then, we have the following.

\begin{proposition}
Let $T>0$. The map $\Theta_T$ defined by (\ref{eq:endpointmap}) is of class $\C^1$. Moreover, for all $(v, \Psi_0) \in \dot{H}_0^1(0,T;\R) \times \mathcal{X}_0,$ we have 
\begin{equation}
\ud \Theta_T (0,\varphi^{\sharp}).(v,\Psi_0) = \left( \Psi_0, \Psi_{|t=T}\right) \in \mathcal{X}_0 \times \mathcal{X}_T,
\label{eq:differential}
\end{equation} where $\Psi$ is the solution of the linearised system around $(0,\varphi^{\sharp})$, i.e.,
\begin{equation}
\left\{ \begin{array}{ll}
i \partial_t \Psi = - \Delta_r \Psi + \dot{v}(t)r^2 \psi_t^{\sharp}, & (t,r) \in (0,T)\times(0,1),\\
\Psi(t,1) = 0, & t\in (0,T), \\
\Psi(0,r) = \Psi_0, & r\in (0,1),
\end{array} \right.
\label{eq:linearisedsystem}
\end{equation} and $(\psi_t^{\sharp})_{t\in(0,T)}$ is given by (\ref{eq:psidiege}). 
\label{proposition:C1}
\end{proposition}

The proof of this result can be carried out as in \cite[Proposition 3, p.531]{KBCL10}, with minor modifications, thanks to Proposition \ref{prop:WP}. We omit the details.

%%%%%%%%%%%%%%%%%%%%%%%%%%%%%%%%%%%%%%%%%%%%%%%%%%%%%%%%%%%%%%%%%%%%%%%%%%%%%%%%%%%%%%%%%%%%%%
%%%%%%%%%%%%%%%%%          LINEARISED SYSTEM

\section{Controllability of the linearised system around $(\varphi^{\sharp}, u\equiv 0)$}
\label{sec: Schrolinearisecontrolable}

The goal of this section is to prove the following result.

\begin{proposition}
Let $T>0$. There exists a continuous linear map 
\begin{equation*}
\begin{array}{llll}
\mathcal{L} :& \mathcal{X}_0 \times \mathcal{X}_T & \rightarrow & \dot{H}^1_0(0,T;\R) \\
   & (\Psi_0, \Psi_f)                   & \mapsto     & v,
\end{array}
\end{equation*} such that for any $\Psi_0 \in \mathcal{X}_0$ and $\Psi_f \in \mathcal{X}_T$, the solution of system (\ref{eq:linearisedsystem}) with initial condition $\Psi_0$ and control $v = \mathcal{L}(\Psi_0,\Psi_f)$ satisfies $\Psi_{|t=T} = \Psi_f$.
\label{prop:linearisedproblemiscontrollable}
\end{proposition} 

The proof of this result relies on the resolution of a suitable moment problem. 
We shall first explain in Section \ref{sec:Heuristicsleadingtoamomentproblem} the heuristics leading to such a moment problem. 
Secondly, we derive in Section \ref{sec:towardstheresolutionofthemomentproblem} the mathematical tools needed to handle 
it, which mainly consist in the construction of a suitable Riesz basis of Nonharmonic Fourier series. 
We finally prove Proposition \ref{prop:linearisedproblemiscontrollable} in Section \ref{sec:resolutionofthemomentproblem} thanks to
the tools developed in the previous sections.

\subsection{Heuristics leading to a moment problem}
\label{sec:Heuristicsleadingtoamomentproblem}

Since (\ref{eq:linearisedsystem}) is a linear system, we may suppose, w.l.o.g., that $\Psi_0 \equiv 0$\footnote{Indeed,
suppose that $\forall \tilde{\Psi}_f \in \mathcal{X}_T$, there exists $v \in \dot{H}_0^1(0,T;\R)$ such that the corresponding solution of (\ref{eq:linearisedsystem}) with 
$\tilde{\Psi}_0 = 0$ satisfies 
$\tilde{\Psi}_{|t=T} = \tilde{\Psi}_f$. Thus, if we are given $\Psi_0 \in \mathcal{X}_0$ and $\Psi_f \in \mathcal{X}_T$, it suffices to choose 
$\tilde{\Psi}^{\sharp} = -e^{-iT \Delta}\Psi_0 + \Psi_f$,
 which provides a control $w \in \dot{H}_0^1(0,T;\R)$, such that the solution $\tilde{\Psi}$ of (\ref{eq:linearisedsystem}) with 
 $\tilde{\Psi}_{|t=0} = 0$ and control $w$ satisfies $\tilde{\Psi}_{|t=T} = \tilde{\Psi}^{\sharp}$. 
Then, $\Psi(t) = e^{-it \Delta} \Psi_0 + \tilde{\Psi}(t)$ satisfies system (\ref{eq:linearisedsystem}) with control $w$, initial datum $\Psi_{|t=0} = \Psi_0$ and verifies $\Psi_{t=T} = \Psi_f$. 
 }. Thus, the solution of system (\ref{eq:linearisedsystem}) admits the following expansion, for any $t\in [0,T]$,
\begin{align}
\Psi(t) & = -i \sqrt{1 - \theta_2 - \theta_3} \sum_{k=1}^{\infty} \int_0^t \dot{v}(s)e^{i(\lambda_k - \lambda_1)s} \ud s \, a_k \varphi_k e^{-i\lambda_k t} \label{eq:heuristique}\\
		& \quad -i \sqrt{\theta_2} \sum_{k=1}^{\infty} \int_0^t \dot{v}(s)e^{i(\lambda_k - \lambda_2)s} \ud s \, b_k \varphi_k e^{-i\lambda_k t} \nonumber \\
		& \quad -i \sqrt{\theta_3} \sum_{k=1}^{\infty} \int_0^t \dot{v}(s)e^{i(\lambda_k - \lambda_3)s} \ud s \, c_k \varphi_k e^{-i\lambda_k t}, \nonumber
\end{align} where $(\lambda_k)_{k\in \N^*}$ and $(\varphi_k)_{k\in \N^*}$ are given by (\ref{eq:radialeigenvectors}) and
\begin{equation}
a_k :=\langle r^2 \varphi_1, \varphi_k \rangle, \, b_k :=\langle r^2 \varphi_2, \varphi_k \rangle, \, c_k :=\langle r^2 \varphi_3, \varphi_k \rangle, \quad \forall k \in \N^*.
\label{eq:abc}
\end{equation}
Given a state $\Psi_f \in \mathcal{X}_T$, for $\mathcal{X}_T$ given by (\ref{X0T}), we look for a control $v \in \dot{H}^1_0(0,T;\R)$ such that 
\begin{equation}
\Psi_{|t=T} = \Psi_f.
\label{eq:desiredfinaleconditionlinearised}
\end{equation} We shall traduce this condition into a trigonometric moment problem as follows.\par

Firstly, since the control $v$ belongs to $\dot{H}^1_0(0,T;\R)$, we must impose
\begin{equation}
 \int_0^T \dot{v}(s) \ud s = 0, \quad \int_0^T s\dot{v}(s) \ud s = 0. \label{eq:heuristiquecontrol}
\end{equation} Next, in order to satisfy equation (\ref{eq:desiredfinaleconditionlinearised}) we shall decompose $\Psi_f$ in Fourier expansion, which yields
\begin{equation}
\Psi_f = \sum_{k=0}^{\infty} \langle \Psi_f, \varphi_k \rangle \varphi_k,
\label{eq:finalstateFourier}
\end{equation} and then rephrase (\ref{eq:desiredfinaleconditionlinearised}) in terms of each Fourier mode. We can do this by separating low and high frequences. \par

\vspace{0.5em}
\textit{High frequencies.} Let $k\geq 4$. We observe that (\ref{eq:desiredfinaleconditionlinearised}) implies, 
according to (\ref{eq:heuristique}) and (\ref{eq:finalstateFourier}), that 
\begin{align*}
ie^{i \lambda_k T} \langle \Psi_f , \varphi_k \rangle = & \sqrt{1 - \theta_2 - \theta_3} a_k \int_0^T \dot{v}(s) e^{i(\lambda_k - \lambda_1)s} \ud s \\
                                                        & \quad \quad + \sqrt{\theta_2} b_k \int_0^T \dot{v}(s) e^{i(\lambda_k - \lambda_2)s} \ud s \\
                                                        & \quad \quad \quad + \sqrt{\theta_3} c_k \int_0^T \dot{v}(s) e^{i(\lambda_k - \lambda_3)s} \ud s.
\end{align*}Since the frequencies $(\lambda_k- \lambda_j)_{k\geq 4}$ for $j=1,2,3$ are distinct (see Proposition \ref{proposition:nonresonance} below), we can prescribe the following
moment values, for any $k\geq 4$,
\begin{align}
\int_0^T \dot{v}(s) e^{i(\lambda_k - \lambda_1)s} \ud s & = \frac{i \sqrt{1 - \theta_2 - \theta_3}}{a_k} \langle \Psi_f, \varphi_k \rangle e^{i \lambda_k T}, \label{eq:momenthighfrequences1} \\ 
\int_0^T \dot{v}(s) e^{i(\lambda_k - \lambda_2)s} \ud s & = \frac{i \sqrt{\theta_2}}{b_k} \langle \Psi_f, \varphi_k \rangle e^{i\lambda_k T}, \label{eq:momenthighfrequences2} \\
\int_0^T \dot{v}(s) e^{i(\lambda_k - \lambda_3)s} \ud s & = \frac{i \sqrt{\theta_2}}{b_k} \langle \Psi_f, \varphi_k \rangle e^{i\lambda_k T}. \label{eq:momenthighfrequences3}
\end{align} \par

\vspace{0.5em}
\textit{Low frequencies.} We observe that $0 \in \left\{ \lambda_k - \lambda_j, \, k,j = 1,2,3  \right\}$. In that case, since the restriction on the control (\ref{eq:heuristiquecontrol})
must be taken into account, we need to separate the low frequences in such a way we could recover the Fourier modes $( \langle \Psi_f, \varphi_k  \rangle )_{k=1,2,3}$
from $\left( e^{i(\lambda_k - \lambda_j)t} \right)_{k,j=1,2,3}$. This can be done by imposing the following conditions
\begin{align}
\int_0^T \dot{v}(s) e^{i(\lambda_2 -\lambda_1)s} \ud s &= \frac{   i \langle \Psi_f, \varphi_2 \rangle e^{i\lambda_2T}  - \sqrt{\theta_3} c_2 \ovl{C}    }{      a_2 \sqrt{1-\theta_2 - \theta_3}}, \label{eq:momentlow1} \\
\int_0^T \dot{v}(s) e^{i(\lambda_3 -\lambda_1)s} \ud s &= \frac{   i \langle \Psi_f, \varphi_3\rangle e^{i\lambda_3T}  - \sqrt{\theta_2} b_3 C           }{    a_3 \sqrt{1-\theta_2 - \theta_3}}, \label{eq:momentlow2} \\
\int_0^T \dot{v}(s) e^{i(\lambda_3 -\lambda_2)s} \ud s &= C, \label{eq:momentlow3} 
\end{align}where $C\in \comp$ satisfies 
\begin{align}
 & 2ib_3 \sqrt{\theta_2 \theta_3} \Re C = \sqrt{1 -\theta_2 - \theta_3} \langle \Psi_f, \varphi_1 \rangle e^{i\lambda_1 T} \ \label{eq:choixdeC} \\
 & \quad \quad \quad \quad\quad\quad \quad \quad \quad  + \sqrt{\theta_2} e^{-i\lambda_2 T} \ovl{\langle \Psi_f, \varphi_2 \rangle} + \sqrt{\theta_3} e^{-i\lambda_3 T} \ovl{\langle \Psi_f, \varphi_3 \rangle}. \nonumber
\end{align} Note that the choice of $C\in \comp$ is possible, since $\Psi_f \in T_{\psi^{\sharp}_{\tau}}\Sph$. \par 

\vspace{0.5em}
\textit{Conclusion.} Putting together (\ref{eq:heuristiquecontrol}), (\ref{eq:momenthighfrequences1})--(\ref{eq:momenthighfrequences3}) 
and (\ref{eq:momentlow1})--(\ref{eq:momentlow3}), we obtain the following moment problem

\begin{equation}
\left\{ \begin{array}{ll}
 & \int_0^T \dot{v}(s) \ud s = 0, \\
 & \int_0^T s \dot{v}(s) \ud s = 0, \\ 
 & \int_0^T \dot{v}(s) e^{i(\lambda_2 -\lambda_1)s} \ud s = \frac{1}{a_2 \sqrt{1-\theta_2 - \theta_3}} \left( i \langle \Psi_f, \varphi_2 \rangle e^{i\lambda_2T}  - \sqrt{\theta_3} c_2 \ovl{C} \right), \\
 & \int_0^T \dot{v}(s) e^{i(\lambda_3 -\lambda_1)s} \ud s = \frac{1}{a_3 \sqrt{1-\theta_2 - \theta_3}} \left( i \langle \Psi_f, \varphi_3\rangle e^{i\lambda_3T}  - \sqrt{\theta_2} b_3 C \right), \\
 & \int_0^T \dot{v}(s) e^{i(\lambda_3 -\lambda_2)s} \ud s = C, \\
 & \int_0^T \dot{v}(s) e^{i(\lambda_k -\lambda_1)s} \ud s =\frac{i\sqrt{1-\theta_2 -\theta_3}}{a_k}\langle \Psi_f, \varphi_k \rangle e^{i\lambda_k T}, \quad \forall k \geq 4, \\
 & \int_0^T \dot{v}(s) e^{i(\lambda_k -\lambda_2)s} \ud s =\frac{i\sqrt{\theta_2}}{b_k}\langle \Psi_f, \varphi_k\rangle e^{i\lambda_k T}, \quad \forall k \geq 4, \\
 & \int_0^T \dot{v}(s) e^{i(\lambda_k -\lambda_3)s} \ud s =\frac{i\sqrt{\theta_3}}{c_k}\langle \Psi_f, \varphi_k \rangle e^{i\lambda_k T}, \quad \forall k \geq 4. 
 \end{array} \right.
 \label{eq:problemedemomentsaresoudre}
 \end{equation} Indeed, if (\ref{eq:problemedemomentsaresoudre}) is satisfied, then (\ref{eq:desiredfinaleconditionlinearised}) holds.

\begin{remark}
At this point, we can justify further the choice of the family of 
states $\varphi^{\sharp}$ given by (\ref{eq:phidiege}). Indeed, choosing, for instance, 
$(\theta_2,\theta_3)=(0,0)\not \in \mathcal{D}$, we get $\varphi^{\sharp} = \varphi_1$. 
However, in this case, the corresponding linearised system around $(0,\varphi_1)$ is not controllable with controls 
in $\dot{H}_0^1(0,T;\R)$, because of the constraint $\int_0^T \dot{v}(s) \ud s =0$.
\label{remark:choixdesetatsnaturels}
\end{remark}

\subsection{Towards the resolution of the moment problem}
\label{sec:towardstheresolutionofthemomentproblem}

The goal of this section is to develop the necessary mathematical tools leading to the proof of Proposition
\ref{prop:linearisedproblemiscontrollable}. In order to do this, we shall 
rewrite the moment problem given by (\ref{eq:problemedemomentsaresoudre}) in an
abstract form that could be handled by the classical results on moment problems consisting in the 
use of Riesz basis (see Appendix \ref{sec: appendixMoment} for details and notation).

\subsubsection{Reinterpretation of the moment problem}
We observe that 
(\ref{eq:problemedemomentsaresoudre}) can be rewritten in the form 
\begin{equation*}
\langle \dot{v}, e^{- i \omega_k s} \rangle_{L^2(0,T; \comp)} = \int_0^T  \dot{v}(s) e^{i \omega_k s} \ud s = d_k,   
\end{equation*} for the family of frequencies
\begin{equation}
\left\{ \omega_k; \, k \in \N \right\} = \left\{  0 \right\} \cup \left\{ j_{0,n}^2 - j_{0,p}^2; \, p=1,2,3, \, n\geq p+1   \right\},
\label{eq:frequences}
\end{equation} rearranged in increasing order and 
\begin{equation}
d_k := \left\{ \begin{array}{ll}
0,   & \textrm{if } k=0, \\
\frac{1}{a_2 \sqrt{1-\theta_2 - \theta_3}} \left( i \langle \Psi_f, \varphi_2 \rangle e^{i\lambda_2T} - \sqrt{\theta_3} c_2 \ovl{C} \right), & \textrm{if }k=1, \\
\frac{1}{a_3 \sqrt{1-\theta_2 - \theta_3}} \left( i \langle \Psi_f, \varphi_3 \rangle e^{i\lambda_3T}  - \sqrt{\theta_2} b_3 C \right), & \textrm{if }k=2, \\
C, & \textrm{if }k=3,\\
\frac{i\sqrt{\theta_3}}{c_k}\langle \Psi_f, \varphi_k \rangle e^{i\lambda_k T}, & \textrm{if }k\in 4 + 3\N, \\
\frac{i\sqrt{\theta_2}}{b_k}\langle \Psi_f, \varphi_k \rangle e^{i\lambda_k T}, & \textrm{if }k\in 5 + 3\N, \\
\frac{i\sqrt{1-\theta_2 -\theta_3}}{a_k}\langle \Psi_f, \varphi_k \rangle e^{i\lambda_k T}, & \textrm{if } k \in 6 + 3\N,
\end{array} \right.
\label{eq:momentproblesegondmembred}
\end{equation} for $C$ given by (\ref{eq:choixdeC}), $\left\{ a_k,b_k,c_k \right\}_{k \in \N^* }$ given by (\ref{eq:abc}) and 
$\left\{ \langle \Psi_f, \varphi_k \rangle  \right\}_{k \in \N^*}$ given by (\ref{eq:finalstateFourier}). Thus, according to Appendix \ref{sec: appendixMoment}, 
let us consider, for a given $T>0$, the family
\begin{equation}
\mathcal{F} := \left\{ t \mapsto e^{ -i\omega_n t}; n \in \N \right\} \subset L^2(0,T;\comp) 
\end{equation} and let us consider the moment set associated to $\mathcal{F}$, i.e.,
\begin{equation*}
 \mathfrak{M}_{L^2(0,T;\comp)}(\mathcal{F}) = \left\{ \left\{ \langle w, e^{ -i\omega_n t} \rangle_{L^2(0,T;\comp)} \right\}_{n\in \N}, \quad w \in L^2(0,T;\comp) \right\}.
\end{equation*} Then, we shall prove that 
\begin{equation*}
\ell_r^2(\N,\comp) := \left\{ \left\{d_k  \right\}_{k\in \N} \in \ell^2(\N,\comp); \, d_0 \in \R  \right\} \subset  \mathfrak{M}_{L^2(0,T;\comp)}(\mathcal{F}).
\end{equation*} More precisely, we have the following result.

\begin{proposition}
Let $  \left\{\omega_k  \right\}_{k\in \N}$ be the increasing sequence defined by (\ref{eq:frequences}). 
Then, for any $T>0$, there exists a continuous linear map 
\begin{equation*}
\mathcal{M} : \R \times \ell^2_r (\N,\comp)  \rightarrow  L^2(0,T;\R), 
\end{equation*} such that for every $\tilde{d}\in\R$ and $d = \left\{d_n \right\}_{n\in \N} \in \ell^2_r(\N,\comp) $, the function $w:=\mathcal{M}(\tilde{d},d)$ satisfies 
\begin{equation} 
\left\{ \begin{array}{l}
\int_0^T w(t) e^{i\omega_k t} \ud t = d_n, \quad \forall n \in \N,  \\
\int_0^T tw(t) \ud t = \tilde{d}. 
\end{array} \right.
\label{eq:theoreticalmomentproblem}
\end{equation}
\label{prop:moment}
\end{proposition}

For the proof of this result, we combine arguments coming from \cite[Appendix B, Corollary 2]{KBCL10} and \cite[Appendix, Proposition 6.1]{NersesyanMorancey}. 
Firstly, we show in Section \ref{sec:anonresonanceproperty} that the frequencies $\left\{ \omega_k  \right\}_{k\in \N^*}$ are non resonant. Secondly,
we construct in Section \ref{sec: constructionofaRieszbasisSchrodinger} a suitable Riesz basis and then we prove Proposition \ref{prop:moment}.

\subsubsection{A non-resonance property}
\label{sec:anonresonanceproperty}
\begin{proposition}
Let $\lambda_n : = j_{0,n}^2$ for any $n \in \N^*$. Then, for any $n,m \in \N^*$ and $p,q \in \left\{ 1,2,3  \right\}$, we have
\begin{equation}
\lambda_n - \lambda_ p \not = \lambda_m - \lambda_ q, \quad \forall n \not = m, \, p \not = q.
\label{eq:nonresonance}
\end{equation}
\label{proposition:nonresonance}
\end{proposition} 

\begin{proof}
Let us assume that $n,m \geq 4$, property (\ref{eq:nonresonance}) being obvious otherwise. \par 
Working by contradiction, let us suppose that there exist $m,n\geq 4$ and $p,q \leq 3$ such that
\begin{equation}
\lambda_n - \lambda_p = \lambda_m - \lambda_q.
\label{eq:resonance}
\end{equation} Moreover, we may assume, without loss of generality, that 
\begin{equation}
n>m>p>q.
\label{eq:mnpq}
\end{equation} We shall distinguish two cases. 

\vspace{0.5em}
\textit{Case 1.} Let us suppose that $p= q +1$. \par 
Then, thanks to (\ref{eq:increasingsequenceofzeros}), we have
\begin{align}
\lambda_n - \lambda_m & = \left( j_{0,n} - j_{0,m} \right)\left( j_{0,n} + j_{0,m} \right) \nonumber \\
					 &	= \sum_{k=m}^{n-1} \left( j_{0,k+1} - j_{0,k} \right)\left( j_{0,n} + j_{0,m} \right) \nonumber \\		
					 & > (n-m) \left( j_{0,p} - j_{0,q} \right) \left( j_{0,n} + j_{0,m} \right). 		\nonumber														
\end{align} Thus, combining this with (\ref{eq:resonance}), we get
\begin{equation*}
\lambda_p - \lambda_q > (n-m) \left( j_{0,p} - j_{0,q} \right) \left( j_{0,n} + j_{0,m} \right).
\end{equation*} This implies
\begin{equation*}
j_{0,p} + j_{0,q} > j_{0,n} + j_{0.m},
\end{equation*} which is incompatible with (\ref{eq:mnpq}), which shows (\ref{eq:nonresonance}) in this case.

\vspace{0.5em}
\textit{Case 2.} Let us suppose that $p=q+2$. \par 
Firstly, let us assume that $n=m+1$. Then, by claim (\ref{eq:resonance}) and using (\ref{eq:increasingsequenceofzeros}) twice, this yields
\begin{align*}
j_{0,m+1} + j_{0,m} & < \frac{(j_{0,p} - j_{0,q})(j_{0,p} + j_{0,q})}{j_{0,m+1}-j_{0,m}} \nonumber \\
& < \frac{(j_{0,p} - j_{0,q})(j_{0,p} + j_{0,q})}{j_{0,p}-j_{0,q+1}} \nonumber \\
& < \left(1 + \frac{j_{0,q+1}-j_{0,q}}{j_{0,p} - j_{0,q+1}} \right) \left(j_{0,p} + j_{0,q} \right) < 2\left(j_{0,p} + j_{0,q} \right).
\end{align*} But this is impossible, since $j_{0,4} + j_{0,3} >2(j_{0,3} + j_{0,1})$, as can be seen from the exact values of these zeros. \par 
Secondly, let us suppose that $n-m\geq 2$, i.e., $\floor{ \frac{n-m}{2} } \geq 1$, where $\floor{ \cdot }$ stands for the floor function. Thus,
\begin{align*}
j_{0,n} - j_{0,m} & > \sum_{i=0}^{\floor{\frac{n-m}{2}}-1} \left( j_{0,m + 2(i+1)} - j_{0,m + 2i} \right) \\
& \geq  \floor{\frac{n-m}{2}} \left( j_{0,p} - j_{0,q} \right) \geq \left( j_{0,p} - j_{0,q} \right).
\end{align*} Thus, (\ref{eq:resonance}) yields
\begin{equation*}
j_{0,n} + j_{0,m} < j_{0,p} + j_{0,q}, 
\end{equation*} which is in contradiction with (\ref{eq:mnpq}).

\end{proof}

\subsubsection{Construction of a Riesz basis}
\label{sec: constructionofaRieszbasisSchrodinger}

\begin{proposition}
Let us set, according to (\ref{eq:frequences}), $\omega_{-n}:= \ovl{\omega_n}$, for any $n\in \N$ and $\omega_0 = 0$. 
Let us define, for a given $T>0$, the families 
\begin{equation*}
\mathcal{F}_T:= \left\{ t \mapsto e^{i\omega_n t}; n \in \Z    \right\} \subset L^2(0,T; \comp), 
\end{equation*} and
\begin{equation*}
\mathcal{F}_T^*:= \left\{ t \mapsto e^{i\omega_n t}; n \in \Z^*    \right\} \subset L^2(0,T; \comp). 
\end{equation*} Then, we have that (see Definition \ref{definition:minimalRiesz} for details)
\begin{enumerate}
 \item $\mathcal{F}_T^*$ is a Riesz basis of $H_T:= Adh_{L^2(0,T;\comp)} \left(span \mathcal{F}_T  \right)$.
 \item $\mathcal{F}_T$ is a minimal family in $L^2(0,T;\comp)$.
\end{enumerate}
\label{proposition:constructionofaRieszbasis} 
\end{proposition}

\begin{proof} \par 
\textit{Step 1.} We prove point (1). \par 
We observe that (\ref{eq:decreasingzeros}) and (\ref{eq:increasingsequenceofzeros}) imply that, for any $k \in \Z^*$,
\begin{align*}
\omega_{k+1} - \omega_{k+1} & = j_{0, N(k+1)}^2 - j_{0, N(k)}^2 \\
			    & = \left( j_{0, N(k+1)} - j_{0, N(k)} \right)\left( j_{0, N(k+1)} + j_{0, N(k)} \right) \\
			    & \geq (j_{0,2} - j_{0,1} ) (j_{0,2} + j_{0,1} ) = j_{0,2}^2 - j_{0,1}^2 > 0,
\end{align*} for some bijection $N : \N^* \rightarrow \N^* $. Thus, (\ref{eq:gapconditionHaraux}) and (\ref{eq:gapconditionHaraux2}) 
are satisfied and Theorem \ref{thm:Haraux} can be applied for any $T \geq \frac{2 \pi}{j_{0,2}^2 - j_{0,1}^2} $. Thus, combining this with
Theorem \ref{thm:abstractRieszbasis}, we deduce that $\mathcal{F}_T^* $ is a Riesz basis of $H_T$ for any $T \geq \frac{2 \pi}{j_{0,2}^2 - j_{0,1}^2} $. \par

Moreover, we notice that, thanks to Beurling's theorem (see Theorem \ref{thm:Beurling}), we can extend the validity of this statement to every $T>0$.  
Let us consider $D^+(\omega)$, the upper density of the sequence $ \left\{ \omega_n \right\}_{n \in \Z}$, according to (\ref{eq:Polyaupperdensity}). 
We shall prove that $D^+(\omega) = 0$. \par 
Indeed, let us observe that (\ref{eq:zerostendtoninfty}) and (\ref{eq:convergingzeros}) imply that
\begin{equation}
\omega_n \rightarrow \infty, \quad \textrm{ as } n \rightarrow \infty.
\label{eq:frequenciestendtoinfinity}
\end{equation} Moreover, we observe that, for a sufficiently large $n_0 \in \N$, 
the frequencies $ \left\{ \omega_n  \right\}_{n\geq n_0}$ can be gathered in successive three-element packets of the form
\begin{equation*}
j_{0,n_0+n}^2 - j_{0,3}^2< j_{0,n_0+n}^2 - j_{0,2}^2 < j_{0,n_0+n}^2 - j_{0,1}^2.
\end{equation*} Consequently, the gap between the elements of each packet must be 
\begin{equation*}
\tilde{\gamma} = \min\left\{ j_{0,3}^2 - j_{0,2}^2, j_{0,2}^2 - j_{0,1}^2  \right\} >0.
\end{equation*} In addition, the gap between the elements of successive packets must be
\begin{equation*}
j_{0,n_0 + n +1 }^2 - j_{0,n_0 + n}^2 + j_{0,1}^2 - j_{0,3}^2, 
\end{equation*} which tends to $\infty$ as $n\rightarrow \infty$, thanks to (\ref{eq:zerostendtoninfty}) and (\ref{eq:convergingzeros}). 
In addition, the non-resonance property (\ref{eq:nonresonance}) ensures that $\omega_k \not = \omega_n$, for any $n \not = k $. We then deduce from this that the frequencies $\omega_k$ do not concentrate, i.e., 
\begin{equation*}
\inf_{k\in \N} \left(  \omega_{k+1} - \omega_k \right)  \geq \tilde{\gamma} > 0.
\end{equation*} On the other hand, let $r>0$ be large enough. According to (\ref{eq:frequenciestendtoinfinity}) and the previous discussion, we must have
\begin{align*}
\max_{I \subset \R \textrm{ interval } \, |I|=r } \# \left\{ \omega_k \in I   \right\} 
& \leq 3 \# \left\{ \omega_k \leq r +j_{0,3}^2  \right\} \\
& \leq 3 \# \left\{ j_{0,k}^2 \leq r +j_{0,3}^2  \right\} \\
& \leq 3 \# \left\{ k^2 \leq r +j_{0,3}^2  \right\} \\
& \leq 3 \sqrt{r + j_{0,3^2}}, 
\end{align*} as $k^2 \leq j_{0,k}^2$ for any $k\in \N^*$, according to (\ref{eq:convergingzeros}) 
and (\ref{eq:increasingsequenceofzeros}). Thus,
\begin{equation*}
D^+(\omega)= \lim_{r\rightarrow \infty} 
\frac{\max_{I \subset \R \textrm{ interval } \, |I|=r } \# \left\{ \omega_k \in I   \right\}}{r} 
 = \lim_{r\rightarrow \infty}  \frac{   3 \sqrt{r + j_{0,3}^2}  }{  r  } =0.
\end{equation*} Theorem \ref{thm:Beurling} allows to conclude. \par

\vspace{0.5em}
\textit{Step 2.} We prove point (2). \par 
Working by contradiction, let us assume that $\mathcal{F}_T$ is not minimal in $L^2(0,T:\comp)$, for some $T>0$. 
Then, the previous step implies (see Remark \ref{remark:Riesz basisareminimal}) that 
\begin{equation*}
t \mapsto t \in Adh_{L^2(0,T;\comp)} \left( span \mathcal{F}_T^*  \right).
\end{equation*} Then, by successive integrations, one checks that
\begin{equation*}
t \mapsto t^j \in Adh_{\C^0([0,T])} \left( span \mathcal{F}_T  \right), \quad \forall j \in \N, \, j\geq 2. 
\end{equation*} On the other hand, the Stone-Weierstrass theorem guarantees that the family defined by 
$\left\{ t \mapsto 1, t \mapsto t^j; j\in\N, \, j\geq 2  \right\}$ is dense in $\C^0([0,T])$. Thus, we deduce that 
\begin{equation}
span \mathcal{F}_T \textrm{ is dense in } L^2(0,T;\comp).
\label{eq:densityinL2}
\end{equation} Let us choose some $\omega \in \R \setminus \left\{ \omega_n  \right\}_{n \in \Z}$. 
The previous step, combined with theorem (\ref{thm:Haraux}), entails that 
$\left\{ t \mapsto e^{i\omega t} \right\} \cup \mathcal{F}_T$ is minimal in $L^2(0,T;\comp)$. But then, we must have
\begin{equation*}
t \mapsto e^{i \omega t } \not \in Adh_{L^2(0,T;\comp)} \left(  span \mathcal{F}_T \right),
\end{equation*} which is a contradiction with (\ref{eq:densityinL2}).\par
\end{proof}

Once we have obtained a suitable Riesz basis, we can prove Proposition \ref{prop:moment}. \par 

\begin{proof}[Proof of Proposition \ref{prop:moment}] 
Let us set $d_k:= \ovl{d_{-k}}$, for any $k\in \Z^*$ with $k<0$. Let $\left\{ \tilde{\xi},  \xi_k,; k \in \Z  \right\}$ be the biorthogonal family to $\mathcal{F}_T$ (see point (3) in
Theorem \ref{thm:abstractRieszbasis}). Using Proposition \ref{proposition:constructionofaRieszbasis} and Remark \ref{remark:Rieszbasissovlemomentproblems},
there exists a constant $C>0$ and a unique $u \in H_T$ satisfying 
\begin{equation*}
\int_0^T u(t) e^{i \omega_k t} \ud t = d_k, \quad \forall k \in \Z,
\end{equation*} and such that 
\begin{equation*}
\| u \|_{L^2(0,T)} \leq C \left(  \sum_{k\in \Z^* }|d_k|^2 \right)^{\frac{1}{2}}.
\end{equation*} Moreover, $u$ is real-valued thanks to the uniqueness. Let us set 
\begin{equation*}
w:=\mathcal{M}(\tilde{d},d) = u + \left( \tilde{d} -\int_0^T t u(t) \ud t \right) \tilde{\xi}. 
\end{equation*} Thus, $w$ solves (\ref{eq:theoreticalmomentproblem}) and is also real-valued,
since $u$ and $\tilde{\xi}$ are so. Moreover, proceeding exactly as in \cite[Corollary 2, Appendix B]{KBCL10}, one shows that the map $\mathcal{M}$ is continuous.
\end{proof}

\subsubsection{A key asymptotic result}  
\label{sec: preuveformuleBessel}
The goal of this section is to prove the following formulae, which are key to prove that $ d = \left\{ d_k \right\}_{k \in \N}$, defined by (\ref{eq:momentproblesegondmembred}) is well-defined and belongs 
to $\ell^2(\N;\comp)$.

\begin{lemma}
For every $\nu \in \N$ and $k,l \in \N^*$ such that $k \not = l$, 
\begin{equation}
\int_0^1 r^3 J_{\nu}(j_{\nu,l}r)J_{\nu}(j_{\nu,k}r) \ud r = \frac{4 j_{\nu,k}j_{\nu,l}J_{\nu+1}(j_{\nu,k})J_{\nu+1}(j_{\nu,l})}{\left( j_{\nu,k}^2 - j_{\nu,l}^2   \right)^2}. \label{eq:Besselsecondformula}
\end{equation}
\label{lemma:asymptotics}
\end{lemma}

\begin{proof}
Let us define, for every $k\in \N^*$,
\begin{equation}
W^2_{\nu,k}(r):= r^2 J_{\nu}(j_{\nu,k}r), \quad \forall r\in (0,1).
\label{eq:W2}
\end{equation} From (\ref{eq:Besseldifferentialequation}), we deduce that $W^2_{\nu,k}$ satisfies the following equation
\begin{equation}
\frac{\ud^2}{\ud r^2} W^2_{\nu,k}(r) -\frac{3}{r}\frac{\ud}{\ud r} W^2_{\nu,k}(r) + \left( j_{\nu,k}^2 + \frac{4-\nu^2}{r^2}\right) W^2_{\nu,k}(r)=0, \quad \forall r\in (0,1).
\label{eq:W2equation}
\end{equation} This implies
\begin{eqnarray}
&&\int_0^1 r^3 J_{\nu}(j_{\nu,k}r)J_{\nu}(j_{\nu,l}r) \ud r = \int_0^1 W^2_{\nu,k}(r) J_{\nu}(j_{\nu,l}r) r \ud r \nonumber \\
&&= -\frac{1}{j_{\nu,k}^2} \int_0^1 \left( \frac{\ud^2 }{\ud r^2} -\frac{3}{r} \frac{\ud}{\ud r} + \frac{4-\nu^2}{r^2} \right)W^2_{\nu,k}(r) J_{\nu}(j_{\nu,l}r) r \ud r \nonumber \\
&&= - \frac{1}{j_{\nu,k}^2} \int_0^1 \left( \frac{\ud^2 }{\ud r^2} + \frac{1}{r}\frac{\ud}{\ud r} - \frac{\nu^2}{r^2} \right) W^2_{\nu,k}(r) J_{\nu}(j_{\nu,l}r) r \ud r  \nonumber\\
&& \quad \quad \quad + \frac{4}{j_{\nu,k}^2}\int_0^1 \left( \frac{1}{r} \frac{\ud}{\ud r} - \frac{1}{r^2}\right) W^2_{\nu,k}(r) J_{\nu}(j_{\nu,l}r) r \ud r. \label{eq:W2integrals}
\end{eqnarray} For the last integral, we have, by (\ref{eq:W2}), (\ref{eq:Besselorthogonality}) and (\ref{eq:derivativeofJnu}), 
\begin{eqnarray}
&& \int_0^1 \left( \frac{1}{r} \frac{\ud}{\ud r} - \frac{1}{r^2}\right) W^2_{\nu,k}(r) J_{\nu}(j_{\nu,l}r) r \ud r  \nonumber\\
&& \quad \quad \quad \quad = j_{\nu,k}\int_0^1 r^2 J_{\nu}'(j_{\nu,k}r) J_{\nu}(j_{\nu,l}r) \ud r \nonumber \\
&& \quad \quad \quad \quad = j_{\nu,k} \int_0^1 r^2 J_{\nu-1}(j_{\nu,k}r) J_{\nu}(j_{\nu,l}r) \ud r. \label{eq:W2integarlsecond}
\end{eqnarray} For the other integral in (\ref{eq:W2integrals}), we have, integrating by parts and using (\ref{eq:Besseldifferentialequation}), 
\begin{equation*}
\int_0^1 \left( \frac{\ud^2 }{\ud r^2} + \frac{1}{r}\frac{\ud}{\ud r} - \frac{\nu^2}{r^2} \right) W^2_{\nu,k}(r) J_{\nu}(j_{\nu,l}r) r \ud r = -j_{\nu,l}^2 \int_0^1 W^2_{\nu,k}(r) J_{\nu}(j_{\nu,l}r) r\ud r.
\end{equation*} Combining this equality with (\ref{eq:W2integrals}) and (\ref{eq:W2integarlsecond}) yields
\begin{equation}
\left( 1 - \frac{j_{\nu,l}^2}{j_{\nu,k}^2} \right) \int_0^1 r^3 J_{\nu}(j_{\nu,k}r)J_{\nu}(j_{\nu,l}r) \ud r = \frac{4}{j_{\nu,k}} \int_0^1 r^2 J_{\nu-1}(j_{\nu,k}r) J_{\nu}(j_{\nu,l}r)  \ud r.
\label{eq:Besselintegral3}
\end{equation} To calculate the last integral, let us define
\begin{equation}
W^1_{\nu-1,k}(r):= r J_{\nu -1}(j_{\nu,k}r), \quad \forall r\in(0,1).
\label{eq:W1}
\end{equation} According to (\ref{eq:Besseldifferentialequation}), we have, for every $r\in (0,1)$,
\begin{equation*}
\frac{\ud^2}{\ud r^2} W^1_{\nu-1,k}(r)- \frac{1}{r} \frac{\ud}{\ud r} W^1_{\nu-1,k}(r) + \left( j_{\nu,k}^2 + \frac{1-(\nu-1)^2}{r^2} \right) W^1_{\nu-1,k}(r) = 0.
\end{equation*} Then,
\begin{eqnarray}
&& \int_0^1 r^2 J_{\nu-1}(j_{\nu,k}r) J_{\nu}(j_{\nu,l}r) \ud r = \int_0^1 W^1_{\nu-1,k}(r) J_{\nu}(j_{\nu,l}r) r \ud r \nonumber \\
&& = -\frac{1}{j_{\nu,k}^2} \int_0^1 \left( \frac{\ud^2}{\ud r^2} -\frac{1}{r}\frac{\ud }{\ud r} + \frac{2\nu -\nu^2}{r^2} \right)  W^1_{\nu-1,k}(r) J_{\nu}(j_{\nu,l}r) r\ud r \nonumber \\
&& = -\frac{1}{j_{\nu,k}^2} \left( \frac{\ud^2}{\ud r^2} + \frac{1}{r}\frac{\ud}{\ud r} - \frac{\nu^2}{r^2} \right) W_{\nu-1,k}^1(r) J_{\nu}(j_{\nu,l}r) r\ud r  \nonumber \\
&& \quad \quad \quad \quad \quad \quad + \frac{1}{j_{\nu,k}^2} \int_0^1 \left( \frac{2}{r} \frac{\ud}{\ud r} - \frac{2\nu}{r^2}  \right) W_{\nu-1,k}^1(r) J_{\nu}(j_{\nu,l}r) r\ud r. \nonumber 
\end{eqnarray} Integrating by parts, and recalling that $J_{\nu}(0)=0$, for any $\nu \in \N^*$, we find 
\begin{align}
& - \frac{1}{j_{\nu,k}^2} \int_0^1 \left( \frac{\ud^2}{\ud r^2} + \frac{1}{r}\frac{\ud}{\ud r} - \frac{\nu^2}{r^2} \right) W_{\nu-1,k}^1(r) J_{\nu}(j_{\nu,l}r) r\ud r \nonumber \\
& \quad \quad = - \frac{1}{j_{\nu,k}^2} \int_0^1 W_{\nu-1,k}^1(r) \left( \frac{\ud^2}{\ud r^2} + \frac{1}{r}\frac{\ud}{\ud r} - \frac{\nu^2}{r^2} \right)J_{\nu}(j_{\nu,l}r) \ud r \nonumber \\
& \quad \quad \quad \quad \quad \quad \quad \quad \quad \quad \quad  + \frac{j_{\nu,l}}{j_{\nu,k}^2} W_{\nu-1,k}^1(1) J_{\nu}'(j_{\nu,l}) \nonumber  \\
& \quad \quad = \left( \frac{j_{\nu,l}}{j_{\nu,k}} \right)^2 \int_0^1 W_{\nu-1,k}^1(r) J_{\nu}(j_{\nu,l}r) r \ud r + \frac{j_{\nu,l}}{j_{\nu,k}^2} J_{\nu-1}(j_{\nu,k}) J_{\nu}'(j_{\nu,l}). \nonumber 
\end{align} This gives
\begin{align}
&\left( 1 - \frac{j_{\nu,l}^2}{j_{\nu,k}^2}  \right) \int_0^1  W_{\nu-1,k}^1(r) J_{\nu}(j_{\nu,l}r) r \ud r \nonumber \\
& \quad \quad  \quad \quad =  \frac{1}{j_{\nu,k}^2} \int_0^1 \left( \frac{2}{r}\frac{\ud}{\ud r} - \frac{2\nu}{r}  \right) W_{\nu-1,k}^1(r)J_{\nu}(j_{\nu,l}r) r \ud r \label{eq:W1integral} \\
& \quad \quad  \quad \quad \quad \quad  \quad \quad  \quad \quad + \frac{j_{\nu,l}}{j_{\nu,k}^2} J_{\nu-1}(j_{\nu,k}) J_{\nu}'(j_{\nu,l}). \nonumber  
\end{align} We treat the last integral separately. Integrating by parts and using (\ref{eq:W1}) and (\ref{eq:derivativeofJnu}), it comes
\begin{eqnarray}
&& \int_0^1 \left( \frac{2}{r} \frac{\ud}{\ud r} - \frac{2\nu}{r^2} \right) W_{\nu-1,k}^1(r) J_{\nu}(j_{\nu,l}r) r \ud r \nonumber\\
&& = -2j_{\nu,l} \int_0^1 W_{\nu-1,k}^1(r) J_{\nu}'(j_{\nu,l}r) \ud r -2\nu \int_0^1 J_{\nu-1}(j_{\nu,k}r) J_{\nu}(j_{\nu,l}r) \ud r \nonumber \\
&& = -2 j_{\nu,l} \int_0^1 r J_{\nu-1}(j_{\nu,k}r) \left[ J_{\nu-1}(j_{\nu,l}r) - \frac{\nu}{j_{\nu,l}r} J_{\nu}(j_{\nu,l}r)  \right] \ud r \nonumber \\
&& \quad \quad \quad \quad -2\nu \int_0^1 J_{\nu-1}(j_{\nu,k}r) J_{\nu}(j_{\nu,l}r) \ud r \nonumber \\
&& = -2j_{\nu,l} \int_0^1 r J_{\nu-1}(j_{\nu,k}r) J_{\nu-1}(j_{\nu,l}r) \ud r. \nonumber
\end{eqnarray} Hence, using (\ref{eq:integralBessel}), we deduce.
\begin{equation*}
\int_0^1 \left( \frac{2}{r} \frac{\ud}{\ud r} - \frac{2\nu}{r^2} \right) W_{\nu-1,k}^1(r) J_{\nu}(j_{\nu,l}r) r \ud r = 0.
\end{equation*} Consequently, from (\ref{eq:W1integral}), we get
\begin{equation*}
 \left( 1 - \frac{j_{\nu,l}^2}{j_{\nu,k}^2}  \right) \int_0^1 W_{\nu-1,k}^1(r) J_{\nu}(j_{\nu,l}r) r \ud r =  \frac{j_{\nu,l}}{j_{\nu,k}^2}  J_{\nu-1}(j_{\nu,k}) J_{\nu}'(j_{\nu,l}).
\end{equation*} Combining this with (\ref{eq:Besselintegral3}) we find

\begin{equation*}
\left( 1 - \frac{j_{\nu,l}^2}{j_{\nu,k}^2}  \right) \int_0^1 r^3 J_{\nu}(j_{\nu,k}r) J_{\nu}(j_{\nu,l}r) \ud r = \frac{4 j_{\nu,l} J_{\nu-1}(j_{\nu,k}) J_{\nu}'(j_{\nu,l})  }{j_{\nu,k} \left( j_{\nu,k}^2 - j_{\nu,l}^2  \right)}
\end{equation*} Hence, this yields (\ref{eq:Besselsecondformula}), since (\ref{eq:derivativeofJnu}) and (\ref{eq:derivativeplusofJnu}) imply that $J_{\nu}'(j_{\nu,k})= -J_{\nu+1}(j_{\nu,k}) $ and $J_{\nu-1}(j_{\nu,l}) = -J_{\nu+1}(j_{\nu,l})$. 
\end{proof}

\subsection{Resolution of the moment problem}
\label{sec:resolutionofthemomentproblem}

\begin{proof}[Proof of Proposition \ref{prop:linearisedproblemiscontrollable}]
We observe that the trigonometric moment problem (\ref{eq:problemedemomentsaresoudre}) can be solved by using Proposition \ref{prop:moment}. 
In order to justify this, we claim that there exist $C_1,C_2,C_3,D_1,D_2,D_3,$ positive constants such that 
\begin{equation}
C_1 \leq j_{0,k}^3 |a_k| \leq D_1 , \quad C_2 \leq j_{0,k}^3|b_k| \leq D_2, \quad C_3 \leq j_{0,k}^3|c_k| \leq D_2, \quad \forall k\in \N^*.
\label{eq:abcasymptotics}
\end{equation} Indeed, let $k > 1$, the case $k=1$ being straightforward. Identity (\ref{eq:Besselsecondformula}) with $\nu =0$, $l=1$, allows to write, through (\ref{eq:abc}) and (\ref{eq:eigenfunctions}), that 
\begin{align}
|a_k| & = \frac{1}{|J_1(j_{0,1})||J_1(j_{0,k})|} \left| \int_0^1 r^3 J_0(j_{0,1}r) J_0(j_{0,k}) \ud r \right| \nonumber \\
& = \frac{4 j_{0,1} j_{0,k}}{\left( j_{0,k} - j_{0,1} \right)^2\left( j_{0,k} + j_{0,1} \right)^2}, \nonumber 
\end{align} and thus, 
\begin{align}
j_{0,k}^3|a_k| & = \frac{4 j_{0,1} j_{0,k}^4}{\left( j_{0,k} - j_{0,1} \right)^2\left( j_{0,k} + j_{0,1} \right)^2}, \nonumber \\
               & \geq 4j_{0,1} \frac{\left( j_{0,k} - j_{0,1}\right)^2}{\left( j_{0,k} - j_{0,1}\right)^2} \frac{j_{0,k}^2}{\left( j_{0,k} + j_{0,k}  \right)^2} \geq j_{0,1}.\nonumber
\end{align} The majoration follows by the same arguments. Then, (\ref{eq:abcasymptotics}) is proved for $\left\{ a_k \right\}_{k\in \N^*}$. 
Let us observe that the other two cases can be done in the same way. \par 

In addition, assumption $\Psi_f \in H^3_{(0),rad}(D,\comp)$, combined with (\ref{eq:abcasymptotics}), gives that $d:= \left\{ d_k \right\}_{k\in \N} \in \ell^2_r(\N;\comp)$. \par 
This allows to apply Proposition \ref{prop:moment}, which provides a function $w := \mathcal{M}(0,d) \in L^2(0,T;\R)$ with 
\begin{equation}
\int_0^T w(\sigma) \ud \sigma =0, \quad \int_0^T \sigma w(\sigma) \ud \sigma = 0.
\label{eq:moyennenulle}
\end{equation} Consequently, setting 
\begin{equation}
t \mapsto v(t):= \int_0^t w(\sigma) \ud \sigma 
\end{equation} we find a control $v \in \dot{H}_0^1(0,T;\R)$ solving (\ref{eq:problemedemomentsaresoudre}). Moreover, the application $(0,\Psi_f) \mapsto v$ thus defined is continuous, thanks to Proposition \ref{prop:moment}.
\end{proof}

%%%%%%%%%%%%%%%%%%%%%%%%%%%%%%%%%%%%%%%%%%%%%%%%%%%%%%%%%%%%%%%%%%%%%%%%%%%%%%%%%%%%%%%%%%%
%%%%%%%%%%%%%%%%%%%%%              PROOF OF THE LOCAL CONTROLLABILITY

\section{Proof of Theorem \ref{thm:controllability2d}}
\label{sec: Schroconclusion}
Following \cite[Section 2.4]{KBCL10}, Theorem \ref{thm:controllability2d} 
is a consequence of the Inverse Mapping Theorem, combining Proposition \ref{prop:linearisedproblemiscontrollable} and Proposition \ref{proposition:C1}.
We omit the details.

%%%%%%%%%%%%%%%%%%%%%%%%%%%%%%%%%%%%%%%%%%%%%%%%%%%%%%%%%%%%%%%%%%%%%%%%%%%%%%%%%%%%%%%
%%%%%%%%%%%%%%%%%%%%%%%%%%%%%%%%%%%%%%%%%%%%%%%%%%%%%%%%%%%%%%%%%%%%%%%%%%%%%%%%%%%%%%%
%%%%%%%%%           COMMENTS AND PERSPECTIVES

\section{Comments and perspectives}
\label{sec: Schrocomments}
In this paper we have proved a controllability result via domain deformations for the Schr\"{o}dinger equation in the unit disc of $\R^2$. This work, the first of this kind in a two-dimensional domain, shows that the geometry of the domain under study is essential. Indeed, our result is possible thanks to the particular geometry of the disc, which allows to exploit the radial symmetry. This yields a simplified situation to which the tools from one-dimensional bilinear control can be adapted. Even if some extensions in this direction are still possible, this feature of our result seems quite limiting.\par 

On the other hand, a major difficulty of this result was to determine the functional framework in which controllability holds. This has been done thanks to a careful analysis of the spectral family given by the Bessel functions. \par  

Any advance in a more general setting would be utterly interesting. The consideration of more general domains and data may lead, very likely, to the use of more general controls, probably space-dependent. Consequently, the tools from bilinear control, very useful in the one-dimensional case and in the present work, will be no longer convenient, in profit of other approaches.

\vspace{2.5em}

\thanks{  \textbf{Acknowledgements:}
The author would like to very much thank Karine Beauchard (Ecole Normale Sup\'{e}rieure de Rennes) for suggesting him this problem and for many fruitful and stimulating discussions and helpful advices. This article has been prepared at the Centre de Math\'ematiques Laurent Schwartz, Ecole polytechnique, Palaiseau, France, as part of the author's PhD dissertation.
}

%%%%%%%%%%%%%%%%%%%%%%%%%%%%%%%%%%%%%%%%%%%%%%%%%%%%%%%%%%%%%%%%%%%%%%%%%%%%%%%%%%%%%%%%
%%%%%%%%%         APPENDIX
%%%%%%%%%%%%%%%%%%%%%%%%%%%%%%%%%%%%%%%%%%%%%%%%%%%%%%%%%%%%%%%%%%%%%%%%%%%%%%%%%%%%%%%%
%%%%%%%%%%%%%%%%%%%%%%%%%%%%%%%%%%%%%%%%%%%%%%%%%%%%%%%%%%%%%%%%%%%%%%%%%%%%%%%%%%%%%%%%
%%%%%%%%%%%%%%%%%%%%%%%%%%%%%%%%%%%%%%%%%%%%%%%%%%%%%%%%%%%%%%%%%%%%%%%%%%%%%%%%%%%%%%%%

\begin{appendix}

\section{Bessel functions}
\label{sec: appendixBessel}

Let $\nu \in \R$. We denote the Bessel function of order $\nu$ of the first kind by $J_{\nu}$ ( see \cite[9.1.10, p.360]{AS}), which satisfies the ordinary differential equation
\begin{equation}
z^2\frac{\ud^2}{\ud z^2}J_{\nu}(z) + z\frac{\ud}{\ud z} J_{\nu}(z) + (z^2 - \nu^2) J_{\nu}(z) = 0, \quad z\in (0,+\infty).
\label{eq:Besseldifferentialequation}
\end{equation}
\subsection{Properties of the zeros}
We denote by $\left\{j_{\nu,k} \right\}_{k \in \N^*}$ the increasing sequence of zeros of $J_{\nu}$, which are real for any $\nu \geq 0$ and enjoy the following properties (see \cite[9.5.2, p.360]{AS} and  \cite[Lemma 7.8, p.135]{KomLor}).
\begin{align}
& \nu < j_{\nu,k} < j_{\nu, k+1}, \quad \forall k \in \N^*,  \label{eq:decreasingzeros} \\
& j_{\nu,k} \rightarrow +\infty, \quad \textrm{ as } k \rightarrow +\infty, \label{eq:zerostendtoninfty} \\
& j_{\nu, k+1} - j_{\nu,k} \rightarrow \pi, \quad \textrm{ as } k \rightarrow \infty. \label{eq:convergingzeros} \\
& \left( j_{0,k+1} - j_{0,k} \right)_{k\in \N^*} \textrm{ is a strictly increasing sequence.} \label{eq:increasingsequenceofzeros}
\end{align}

\subsection{Integral identities}
We also have the integral formulae (\cite[11.4.5, p.485]{AS})
\begin{equation}
\int_0^1 r J_{\nu}(j_{\nu,l}r) J_{\nu}(j_{\nu,k}r) \ud r = \frac{1}{2}|J_{\nu +1}(j_{\nu,k})|^2 \delta_{l,k}, \quad \forall l,k \in \N^*. 
\label{eq:Besselorthogonality} 
\end{equation} and (see \cite[11.3.29, p. 484]{AS})
\begin{equation}
(\alpha^2 - \beta^2)\int_0^1 r J_{\nu}(\alpha r) J_{\nu}(\beta r) \ud r = \alpha J_{\nu+1}(\alpha)J_{\nu}(\beta) - \beta J_{\nu}(\alpha)J_{\nu+1}(\beta), 
\label{eq:integralBessel}
\end{equation} for any $\alpha, \beta\in \R,$ with $ \alpha \not = \beta$. We have the differential relations (see \cite[9.1.27, p. 361]{AS})
\begin{eqnarray}
&& J_{\nu}'(r) = -J_{\nu+1}(r) +\frac{\nu}{r}J_{\nu}(r), \quad r \in (0,+\infty),\label{eq:derivativeplusofJnu} \\
&& J_{\nu}'(r) = J_{\nu-1}(r) -\frac{\nu}{r}J_{\nu}(r), \quad r \in (0,+\infty). \label{eq:derivativeofJnu} 
\end{eqnarray}

\section{Moment problems}
\label{sec: appendixMoment}
In this section we gather some classical material concerning abstract moment problems in Hilbert space and trigonometric moment problems in $L^2(0,T;\comp)$, that have been used in section \ref{sec: Schrolinearisecontrolable}.  

\subsection{Abstract moment problems}
Let $ H$ be a separable Hilbert space, equipped with the scalar product $\langle \cdot , \cdot \rangle_H$,  and let $\mathcal{S} = \left\{ f_k \right\}_{k\in \Z}\subset H$ be a family of elements of $H$. Given a sequence of complex numbers $\left\{ c_k \right\}_{k\in \Z}$, we want to determine whether the moment problem
\begin{equation}
\langle f, f_k \rangle_H = c_k, \quad \forall k \in \Z,
\label{eq:abstractmomentproblem}
\end{equation}can be solved for some element $f\in H$. In particular we study the moment set associated to $\mathcal{S}$, which is defined (see \cite[Ch.4 Sect.2, p.128]{Young}) by
\begin{equation*}
\mathfrak{M}_H(\mathcal{S}) := \left\{ \left\{ \langle g, f_k \rangle_H \right\}_{k\in \Z}; \quad g\in H \right\} \subset \comp^{\Z}.
\end{equation*} Let us notice that, in practice, we are interested in solving the moment problem (\ref{eq:abstractmomentproblem}) 
for $\left\{ c_j \right\}_{k\in \Z}\in \ell^2(\Z;\comp)$, i.e., we want $\ell^2(\Z;\comp) \subset \mathfrak{M}_H(\mathcal{S})$. 
This necessitates some conditions on the family $\mathcal{S}$, that we briefly describe below.

\begin{definition}
Let $H$ be a separable Hilbert space and let $\mathcal{S} = \left\{ f_k \right\}_{k\in \Z}\subset H$. Then,
\begin{enumerate}
\item $\mathcal{S}$ is a minimal family in $H$ if 
\begin{equation*}
\forall j \in \Z, \quad f_j \not \in Adh_H \left( span \left\{ f_k; \quad k\not = j \right\}\right),
\end{equation*}

\item $\mathcal{S}$ is a Riesz basis of $H$ if there exists an orthonormal basis of $H$, say $\left\{ e_k \right\}_{k\in \Z}$, and a linear and bounded mapping $T \in \mathscr{L}(H)$ which is invertible and satisfies 
\begin{equation*}
Te_k = f_k, \quad \forall k \in \Z.
\end{equation*} 

\item $\mathcal{S}$ satisfies the Riesz-Fischer property in $H$ if 
\begin{equation*}
\ell^2(\Z;\comp) \subset \mathfrak{M}_H(\mathcal{S}).
\end{equation*}
\end{enumerate}

\label{definition:minimalRiesz}
\end{definition}

\begin{remark}
Observe that it follows from the previous definition that any Riesz basis of a separable Hilbert space $H$ is also a minimal family in $H$.
\label{remark:Riesz basisareminimal} 
\end{remark}

The following result provides two powerful criteria to check whether a given family $\mathcal{S} \subset H$ 
is a Riesz basis or satisfies the Riesz-Fischer property in $H$.

\begin{thm}
Let $H$ be a separable Hilbert space and let $\mathcal{S} = \left\{ f_k \right\}_{k\in \Z}\subset H$. Then,
\begin{enumerate}

\item \cite[Ch.4 Sect.2. Th.3, p.129]{Young} $\mathcal{S}$ satisfies the Riesz-Fischer property if there exists a constant $m>0$ such that the inequality
\begin{equation*}
m \sum_{k\in \Z} |c_k|^2 \leq \left\| \sum_{k\in \Z} c_k f_k   \right\|^2_H,
\end{equation*} holds for any $\left\{c_k\right\}_{k\in \Z} \subset \comp^{\Z}$ with finite support.

\item \cite[Ch.1 Sect.8. Th.9, (3) p.27]{Young} $\mathcal{S}$ is a Riesz basis if there exist $M,m>0$ such that the inequality
\begin{equation}
m \sum_{k\in \Z} |c_k|^2 \leq \left\| \sum_{k\in \Z} c_k f_k   \right\|^2_H \leq M \sum_{k\in \Z} |c_k|^2
\label{eq:Rieszcriterium}
\end{equation} holds for any $\left\{c_k\right\}_{k\in \Z} \subset \comp^{\Z}$ with finite support.

\item \cite[Ch.1 Sect.8. Th.9, (5) p.27]{Young} $\mathcal{S}$ is a Riesz basis if and only if 
one has $Adh_H \left( span S \right) = H$\footnote{In this case 
the familiy is called complete in $H$ (see \cite[Ch.1, Sect.5 p.16]{Young}).} and there exists a family, say 
$S^{\perp} = \left\{ g_k  \right\}_{k \in \Z^*} \subset H$, satisfying that $Adh_H \left(span S^{\perp} \right) = H$ and such that\footnote{Such
 a family $S^{\perp}$ is called biorthogonal to S (see \cite[Ch.1 Sect.5 p.24]{Young}).} 
\begin{equation*}
 \langle g_n , f_k \rangle = \delta_{n,k}, \quad \forall n,k \in \Z.
\end{equation*}
\end{enumerate}
\label{thm:abstractRieszbasis}
\end{thm}

\begin{remark}
The previous result shows that if $\mathcal{S}$ is a Riesz basis of $H$, 
then it satisfies the Riesz-Fischer property in $H$. Thus, in particular, $\ell^2(\Z,\comp) \subset \mathfrak{M}_H(\mathcal{S})$.  
This allows to deduce that if $S$ is a Riesz basis of $H$, then the moment problem (\ref{eq:abstractmomentproblem}) can be solved in $H$ for a given 
$\left\{ c_k \right\}_{k \in \Z} \in \ell^2(\Z; \comp)$.
\label{remark:Rieszbasissovlemomentproblems}
\end{remark}

\subsection{Trigonometric moment problems}
Let us focus next on the choice $H=L^2(0,T;\comp)$, for some $T>0$. Let us consider families given in the form
\begin{equation*}
\mathcal{S} = \left\{ t \mapsto e^{i\omega_k t}, \, k \in \Z   \right\} \subset L^2(0,T;\comp), \quad \textrm{with} \quad \left\{ \omega_k   \right\}_{k\in \Z} \subset \R^{\Z}.
\end{equation*} In order to determine if such a family $\mathcal{S}$ is a Riesz basis of $H$, it is crucial to analyse the separation properties of 
the frequences $\left\{ \omega_k   \right\}_{k\in \Z}$, as this allows to fulfill (\ref{eq:Rieszcriterium}) through an 
Ingham-type inequality (see \cite{Ingham}). We shall recall next a classical result due to A.Haraux (see \cite{Haraux} and \cite[Sect. 4.4, p.69]{KomLor} for a proof).

\begin{thm}[Haraux]
Let $N\in \N$ and let $\left\{ \omega_k   \right\}_{k\in \Z}$ be an increasing sequence of $\R^{\Z}$ such that the following gap conditions
\begin{align}
\omega_{k+1} - \omega_k \geq \gamma >0, \quad \forall k \in \Z, \textrm{ with }|k|\geq N, \label{eq:gapconditionHaraux}	\\
\omega_{k+1} - \omega_k \geq \rho >0, \quad \forall k \in \Z, \label{eq:gapconditionHaraux2}
\end{align} are satisfied for some $\gamma$ and $\rho$. Let
\begin{equation}
T \geq \frac{2\pi}{\gamma}.
\label{eq:tempsminimal}
\end{equation} Then, there exist $M,m>0$ such that the inequality
\begin{equation}
m \sum_{k\in \Z}|c_k|^2 \leq \int_0^T \left| \sum_{k\in \Z} c_k e^{i\omega t}  \right|^2 \ud t \leq M \sum_{k\in \Z}|c_k|^2,
\label{eq:HarauxIngham}
\end{equation} holds for any sequence $\left\{ c_k \right\}_{k\in \Z} \subset \comp^{\Z}$ with finite support. 
\label{thm:Haraux}
\end{thm} Condition (\ref{eq:tempsminimal}) can be sharpened by using the following result, due to A. Beurling 
(see \cite[Th.9.2, p.174]{KomLor}).

\begin{thm}[Beurling]  
Let $\omega = \left\{ \omega_k \right\}_{k\in \Z} \subset \R^{\Z}$ satisfying (\ref{eq:gapconditionHaraux}) and (\ref{eq:gapconditionHaraux2}). 
Then, the real number 
\begin{equation}
D^+(\omega) := \lim_{r\rightarrow \infty} \frac{\max_{I \subset \R \textrm{ interval } \, |I|=r } \# \left\{ \omega_k \in I   \right\}}{r}, 
\label{eq:Polyaupperdensity}
\end{equation} called the upper density of $\omega$, is well-defined and inequality (\ref{eq:HarauxIngham}) holds for any $T \geq 2\pi D^+(\omega)$.
\label{thm:Beurling}
\end{thm}

\end{appendix}

\end{document}